\documentclass[sn-mathphys-num,lineno]{sn-jnl}% Math and Physical Sciences Numbered Reference Style 
\usepackage{epstopdf}

\ifpdf
\DeclareGraphicsExtensions{.eps,.pdf,.png,.jpg}
\else
\DeclareGraphicsExtensions{.eps}
\fi
\newcommand{\RN}[1]{%
	\textup{\uppercase\expandafter{\romannumeral#1}}%
}

\usepackage{anyfontsize}
\usepackage{graphicx}%
\usepackage{multirow}%
\usepackage{amsmath,amssymb,amsfonts}%
\usepackage{amsthm}%
\usepackage{mathrsfs}%
\usepackage[title]{appendix}%
\usepackage{xcolor}%
\usepackage{textcomp}%
\usepackage{manyfoot}%
\usepackage{booktabs}%
\usepackage{algorithm}%
\usepackage{algorithmic}
\usepackage{listings}%
\newtheorem{theorem}{Theorem}%  meant for continuous numbers

\newtheorem{lemma}[theorem]{Lemma}
\newtheorem{remark}{Remark}%

\newtheorem{definition}{Definition}%
\allowdisplaybreaks

\begin{document}

\title[Article Title]{Incremental data compression for PDE-constrained optimization with a data assimilation application}


\author[1]{\fnm{Xuejian} \sur{Li}}\email{xuejial@clemson.edu}

\author*[2]{\fnm{John R.} \sur{Singler}}\email{singlerj@mst.edu}

\author[2]{\fnm{Xiaoming} \sur{He}}\email{hex@mst.edu}
%\equalcont{All authors contributed equally to this work.}

\affil[1]{\orgdiv{School of Mathematical and Statistical Sciences}, \orgname{Clemson University}, \orgaddress{\city{Clemson}, \postcode{29364}, \state{SC}, \country{USA}}}

\affil[2]{\orgdiv{Department of Mathematics and Statistics}, \orgname{Missouri University of Science and Technology}, \orgaddress{\city{Rolla}, \postcode{65409}, \state{MO}, \country{USA}}}


\abstract{
		We propose and analyze an inexact gradient method based on incremental proper orthogonal decomposition (iPOD) to address the data storage difficulty in time-dependent PDE-constrained optimization, particularly for a data assimilation problem as a detailed demonstration for the key ideas. The proposed method is proved robust by rigorous analysis. We first derive a sharp data compression error estimate of the iPOD with the help of Hilbert-Schmidt operators. Then we demonstrate a numerical PDE analysis to show how to properly choose the Hilbert space for the iPOD data compression so that the gradient error is under control. We further prove that for a convex problem with appropriately bounded gradient error, the inexact gradient method achieves the accuracy level of the optimal solution while not hurting the convergence rate compared with the usual gradient method. Finally, numerical experiments are provided to verify the theoretical results and validate the proposed method.
}




\maketitle

\section{Introduction}\label{sec1}

Time-dependent PDE-constrained optimization is commonly witnessed in many applications, such as optimal control \cite{MR1871460, MR1759904, lewisvrabiesyrmos2012, MR0271512, MR0204181}, data assimilation \cite{altafelgharamtiheeminkhoteit2013,MR3062044,MR2289211,MR3345226}, and other inverse identification problems \cite{MR1861508,epabielak2008,MR1991790,hattasmartinstadlerwilcox2012,MR4242001,MR3676924,MR3465299}. Gradient-based methods, such as steepest descent, conjugate gradient, and quasi-Newton methods, are the central machinery to solve this type optimization problem, in which calculating the gradient is a key step. To obtain the gradient, one usually needs to solve a forward time-dependent PDE and store all forward solutions for later use in solving a so-called backward adjoint equation. However, during this procedure, the storage of forward solution data is prohibitive when large-scale temporal and spatial simulations are involved \cite{Gl2022,Hu2023}, which makes the gradient calculation very difficult.  

There are multiple candidates to deal with this data storage difficulty, such as checkpoint \cite{MR2520289,MR1856306,griewank_1992,griewank_walther_2000}, data compression \cite{Margcompress2021,Muthucompress2021,MR4026164,LiiSVD2021}, and others \cite{GotsReMemory2015}. 
The checkpoint method is frequently used to reduce the memory requirement for PDE-constrained optimization; however, this method may largely increase the computation time due to many recalculations \cite{MR1856306,heuveline_walther_2006}. 
In contrast, data compression, such as proper orthogonal decomposition (POD) \cite{brand2002,MR3775096,MR4017489,MR3986356}, can save memory effectively for large-scale low-rank data but only introduce minimal extra computational time, %(compared to the optimization time), 
which is attractive. The POD 
is used to handle memory limitations for optimization in the literature \cite{MR4026164,LiiSVD2021}. However, it is still new to this field and under-investigated. This paper will thoroughly consider POD as a data compression technique to address the data storage issue for time-dependent PDE-constrained optimization. 

%Proper orthogonal decomposition (POD) is a method attempting to find an optimal low-dimensional basis to best approximate a set of given data. The optimal basis is called POD modes. 
The POD attempts to find an optimal low-dimensional basis to best approximate a set of given data, which is usually solved by singular value decomposition (SVD),
% \begin{align}\label{svdi}
	% U^{m\times n}\approx V^{m\times r}\Sigma^{r\times r}\left(W^{n\times r}\right)^T, ~~r\||n,
	%\end{align}
	%Here, $m,n$ and $r$ are positive integers, and $V^{m\times r}\Sigma^{r\times r}W^{r\times n}$ is the first $r^{th}$ order truncated SVD of $U^{m\times n}$. The optimal basis (or POD modes) consists of the column vector of $ V^{m\times r}$.
	and the optimal basis (POD modes) consists of the dominant singular vectors.
	The POD is often referred to as a model order reduction approach to simplify the structure of complicated systems, which has shown numerous applications in simulating ODEs, PDEs, and many other science problems, see e.g., \cite{MR1868765,MR2136757,MR1921667,MR2387420,MR2044719,MR1204279,MR3800253,MR3809542,MR3672145,MR3608749,MR4159637,MR2998478,MR3345226}. 
	Also, due to the low dimension representation, POD can be considered as a data compression technique to optimize the storage of large-scale data and address computer memory limitations. %which will be the main interest of this paper to deal with the aforementioned forward-backward PDE solving in optimization problem.   
	%As shown in (\ref{svdi}), POD has close connection to SVD and one key step of POD is to compute the SVD of the given data matrix. There are many methods, such as in R, Python, and Matlab, to quickly calculate the SVD of a moderate size data matrix, which require a full storage of given data matrix first and compute the SVD then, and are called offline method. However, the data size are growing rapidly nowadays, offline method is impractical in lots of scenarios because of computer memory limitation. 
	Unfortunately, many SVD computing methods, such as those in R, Python, and Matlab, are offline and require storing the entire data matrix first before computing the SVD, which still requests large computer memory. To overcome this dilemma, researchers devoted many efforts to figuring out other reasonable SVD computation algorithms \cite{stange2008,MR3428564,MR3340151,MR3594691,MR3912526,MR3341249,MR3763927}. The iPOD (or iSVD) is one of the successful algorithms that was first proposed by Brand \cite{brand2002} and afterward significantly extended and improved by others \cite{MR3775096,MR4017489,MR3986356,Yw2022}. The iPOD uses an online projection to process data column by column and incrementally update the POD until the last column data is processed. Attractively, the used column data is no longer kept and thus there is no necessity to store the data matrix.  
	This computing property is ideal to handle the data storage difficulty for time-dependent PDE-constrained optimization.
	
	%The main interest of this paper is to introduce iPOD to save computer memory for optimization problems.
	Based on a data assimilation problem as the demonstration example, we utilize the iPOD to incrementally produce the POD of the forward PDE solution as it is numerically simulated. Again, one key feature of this step is that once the solution has been used to update the POD it no longer needs to be stored, which allows us to effectively compress the large-scale solution data. We then decompress the data one by one as it is needed for solving the backward adjoint equation. In this way, the data storage difficulty is addressed and the gradient calculation is possible. %thereby driving the gradient-based optimization iterations towards to optimal solution. 
	This approach is an inexact gradient method since the compressed (approximated) data is involved during calculation, and errors will be introduced to the gradient. 
	%Note that there will be error introduced to gradient since the compressed or approximated forward solution is used. Therefore, we call this inexact gradient method. 
	
	To rigorously regulate the gradient error and demonstrate the robustness of this method, we provide a theoretical analysis. First, we provide an accurate error estimate between the exact and compressed data while iPOD is utilized. The properties of Hilbert-Schmidt operators play a fundamental role in estimation. Second, with suitable Hilbert space for the iPOD, we show that the gradient error introduced by iPOD is under control. To provide insight on properly applying iPOD in different scenarios, the gradient error analysis will be presented based on a linear and a nonlinear PDE constraint, respectively. Third, with an appropriate iPOD setup, we prove that the inexact gradient method reduces memory requirements without hurting the convergence rate and accuracy of the optimal solution compared to the usual gradient method. 
	
	The paper is organized as follows. In Section \ref{Background}, we recall preliminary knowledge. In Section \ref{Inexactgradientmethod}, we propose the inexact gradient method based on iPOD data compression for time-dependent PDE-constrained optimization. In Section \ref{SDA}, we prove that the inexact gradient method achieves the optimal solution with similar computational time compared to usual gradient methods. In Section \ref{NE}, we provide numerical tests to validate the proposed method. Finally, we draw conclusions in Section \ref{Conclusion}.

	\section{Preliminaries}\label{Background}
	This section introduces preliminaries of POD, Hilbert-Schmidt operators, and iPOD (or iSVD).
	\subsection{POD and Hilbert-Schmidt operators}
	% POD has close connection to SVD, thus we first review the SVD of linear compact operator. 
	
	Denote the norm of a Hilbert space $X$ as $\|\cdot\|_X$ that is induced by inner product $(\cdot,\cdot)_X$. If $X$ is a finite $m$-dimensional Hilbert space, then $X$ can be viewed as a $m$-dimensional weighted Euclidean space $\mathbb{R}^m_M$, i.e., $\forall x,y\in X$, $(x,y)_X=(\vec {x},\vec{y})_{\mathbb{R}_M^m}=\vec{y}^TM\vec{x}$, where $\vec {x}$ and $\vec{y}$ denote the vector representation of $x$ and $y$ under a given basis of $X$, $\vec{y}^T$ denotes the transpose of $\vec{y}$, and $M$ is a positive definite weight matrix. We denote $\mathscr{L}(X,Y)$ as the collection of linear bounded operators that map $X$ to $Y$. In addition, we denote $\langle\cdot,\cdot\rangle$ as a generic duality pair between a Hilbert space and its dual space.

	The POD (or optimal data compression) problem is described as follows: Given a set of functional data $\{u_j\}_{j=1}^n\subset X$ and a value $ r $, $1\leq r \ll n$, find the $r$ optimal orthonormal basis vectors $\{\Phi_j\}_{j=1}^r\subset X $ % and  coefficients $\{\alpha_{ij}\}_{i=1,1}^{r,k}\in \mathbb{R}$ 
	such that the data compression error $E_r$ is minimized:
	\begin{align}\label{PODproblem}
		%\arg \min_{\{\phi_i\}_{i=1}^r\subset X, \alpha_{ij}\in \mathbb{R}}
		E_r=\sum_{j=1}^n\|u_j-\Pi_{X_r}u_j\|_X,%~~\text{is minimized,}
	\end{align}
	where $\Pi_{X_r}:X\mapsto X_r$ is the orthogonal projection onto $X_r=\mathrm{span}\{\Phi_j\}_{j=1}^r\subset X$. 
	%In order to solve problem (\ref{PODproblem}), we introduce a linear bounded operator
	In order to solve problem (\ref{PODproblem}), we introduce the POD operator $\mathcal{U}: \mathbb{R}^n\mapsto X$ defined by
	\begin{align}\label{svdoperator}
		\mathcal{U}z=\sum_{j=1}^nz_j u_j,~~ z=\begin{bmatrix}z_1&z_2& \cdots & z_n
		\end{bmatrix}^T\in \mathbb{R}^n.%~~\text{and~~} c_i\in X.
	\end{align}
	The operator $\mathcal{U}: \mathbb{R}^n\mapsto X$ is finite rank and thus compact and has a SVD:  $\mathcal{U}z=\sum_{j=1}^s \sigma_j(z,\eta_j)_{\mathbb{R}^n}\psi_j$, $\forall z\in \mathbb{R}^n$, $s\leq n$. The $\{\sigma_j\}_{j=1}^s$ are the non-zero singular values of $\mathcal{U}$, the $\{\psi_j\}_{j=1}^s\in X$ and $\{\eta_j\}_{j=1}^s\in \mathbb{R}^n$ are the corresponding singular vectors. It turns out that problem (\ref{PODproblem}) is equivalent to finding the best $r$ rank Hilbert-Schmidt approximation of $\mathcal{U}$ \cite{John2014,Siro1987}:
	\begin{align}\label{H1}
		\min_{K\in \mathscr{L}(\mathbb{R}^n, X),~ \mathrm{rank}(K)=r} \| \mathcal{U}-K \|_\mathrm{HS}=\sum_{j=1}^n\|u_j-\Pi_{X_r}u_j\|_X.
	\end{align}
	%From operator approximation theory, 
	The operator $K$ is obtained by the first $r^{th}$ order truncated SVD of $\mathcal{U}$: $\mathcal{U}_rz=\sum_{j=1}^r \sigma_j(z,\eta_j)_{\mathbb{R}^n}\psi_j$ $\forall z\in \mathbb{R}^n$, $r\leq s $, and the singular vectors $\{\psi_j\}_{j=1}^r$ are the $r$ optimal basis vectors $\{\Phi_j\}_{j=1}^r$ to solve problem (\ref{PODproblem}). The optimal truncation error (or data compression error) $E_r$ is equal to:
	\begin{align}\label{H2}
		E_r= \sum_{j=1}^n\|u_j-\Pi_{X_r}u_j\|_X^2=\|\mathcal{U}-\mathcal{U}_r\|_\mathrm{HS}=\sum_{j\geq r+1 }\sigma_j^2. 
	\end{align}
	For more understanding of SVD, POD, and Hilbert-Schmidt operators, one can refer to \cite{gohberG1990,MR751959}.
	
	As illustrated in (\ref{H1})-(\ref{H2}), 
	%where $c_i^r=\sum_{j=1}^r\alpha_{ij}\phi_j,~i=1,2,3,\cdots,k$. 
	Hilbert-Schmidt operators are important machinery for solving the POD problem. We here briefly introduce Hilbert-Schmidt operators and their properties, which will be useful for the analysis later.  
	
	\begin{definition}\label{HSd}
		Let $X$ and $Y$ be separable Hilbert spaces. An operator $B\in \mathscr{L}(X,Y)$ is Hilbert-Schmidt if
		\begin{align}
			\sum_{j\geq 1}\|Bx_j\|^2_Y<\infty~~\text{for any total orthonormal sequence } \{x_j\}_{j\geq 1}\subset X.
		\end{align}
		The Hilbert-Schmidt norm of $B$ is defined by $\|B\|_\mathrm{HS}=(\sum_{j\geq 1}\|Bx_j\|^2_Y)^{\frac{1}{2}}$.
	\end{definition}

	\begin{lemma}\label{SVDC}\cite[Section \RN{5}]{DavisFun}
		Let $X$ and $Y$ be separable Hilbert spaces. If an operator $B\in \mathscr{L}(X,Y)$ is compact and its singular values $\{\sigma_j\}_{j\geq 1}$ satisfy $\sum_{j\geq 1} \sigma_j^2<\infty$, then $B$ is Hilbert-Schmidt and $\|B\|_\mathrm{HS}^2=\sum_{j\geq 1} \sigma_i^2$. 
	\end{lemma}
	\vspace{1mm}
	\begin{remark}
		Hilbert-Schmidt operators are compact, but not all compact operator are Hilbert-Schmidt.
	\end{remark}
	\vspace{1mm}
	
	Since $\mathcal{U}$ defined by (\ref{svdoperator}) is finite rank and thus Hilbert-Schmidt and compact, one can further verify the following Lemma with help of Definition \ref{HSd}.
	\vspace{1mm}
	\begin{lemma}\label{HSnorm}\cite[Lemma A.3]{Sarah20}
		$\mathcal{U}$ is Hilbert-Schmidt and $\|\mathcal{U}\|_\mathrm{HS}^2=\sum_{j=1}^{n}\|u_j\|_X^2$.
	\end{lemma}
	\vspace{1mm}
	%For more detail and understanding of Hilbert-Schmidt operator, reader can refer \cite{gohberG1990,MR751959,an2022functional}. 

	Since Range($\mathcal{U}$) belongs to a finite dimension Hilbert space $X=\mathbb{R}^m_M$ with prescribed basis, the operator $\mathcal{U}$ can be represented by the matrix $ U: \mathbb{R}^n\mapsto \mathbb{R}^m_M $ given by $U^{}=\begin{bmatrix}
		\vec{u}_1&\vec{u}_2&\cdots&\vec{u}_n
	\end{bmatrix}$. In this sense, we denote $\|U\|_\mathrm{HS}=\|\mathcal{U}\|_\mathrm{HS}$. The Hilbert adjoint operator of $U: \mathbb{R}^n\mapsto \mathbb{R}^m_M$ is $U^{*}: \mathbb{R}^m_M\mapsto \mathbb{R}^n$ given by $U^{*}=U^TM$, which is verified via $(U\vec{x},\vec{y})_{\mathbb{R}^m_M}=\vec{y}^TMU\vec{x}=(U^TM\vec{y})^T\vec{x}=(\vec{x},U^{*}\vec{y})_{\mathbb{R}^n}$, $\forall \vec{x}\in \mathbb{R}^m_M$, $\vec{y}\in \mathbb{R}^n$. Thus, for different Hilbert space $X$ or weight matrix $M$, the SVD of $U$ will be different. In general, we define the SVD of $U:  \mathbb{R}^n\mapsto \mathbb{R}^m_M$ as follows:
	\vspace{1mm}
	%We formally read that:
	%We formally define the weighted SVD of matrix $C$ as follows. 
	\begin{definition}\label{WSVD}
		The core $M$-weighted SVD of a matrix $U_{m\times n}:\mathbb{R}^n\mapsto \mathbb{R}_M^m$ is a decomposition $U=V\Sigma W^T$, where $V\in \mathbb{R}^{m\times s}$, $\Sigma\in \mathbb{R}^{s\times s}$, and $W\in \mathbb{R}^{n\times s}$ satisfy
		\begin{align}\label{ws}
			V^TMV=I,~~\Sigma=\mathrm{diag}(\sigma_1,\sigma_2,\cdots, \sigma_s),~~W^TW=I.
		\end{align}
		Here, $\sigma_1\geq\sigma_2\geq\cdots\geq\sigma_s>0$ are positive singular values of $U$, and the column of $V$ and $W$ are the left and right singular vectors of $U$.  
	\end{definition}
	
	In Definition \ref{WSVD}, we follow \cite[Definition 2.1]{MR3775096} and use the terminology ``core'' to exclude the zero singular values since POD does not need information from zero singular values for data compression. We refer to the SVD of $U: \mathbb{R}^n\mapsto \mathbb{R}^m$ ($M=I$) as the core standard SVD and the SVD of $U: \mathbb{R}^n\mapsto \mathbb{R}^m_M$ ($M\neq I$)  as the core $M$-weighted SVD, in order to distinguish the SVD of $U$ with standard and weighted Euclidean range spaces. Also, we refer to the matrix $W$ in (\ref{ws}) satisfying $W^TW=I$ as a standard orthonormal matrix and the matrix $V$ satisfying $V^TMV=I$ as an $M$-orthonormal matrix, respectively.  
	\vspace{1mm}
	
	% Note that $U=V\Sigma W^T$ in Definition \ref{WSVD} is a reduced SVD, not including zero singular value. 
	Based on SVD and equations (\ref{PODproblem})-(\ref{H2}), 
	what POD attempts to obtain is:
	\begin{align}
		&	U_{m\times n}\approx V_{m\times r}\Sigma_{r\times r}(W_{n\times r})^T ~~r<s ~\text{with}~r\ll n,~\text{and}\\
		&	V^TMV=I,~\Sigma_r=\mathrm{diag}(\sigma_1,\sigma_2,\cdots, \sigma_r),~W^TW=I.
	\end{align}
	Compared to (\ref{ws}), which removes the zero singular values, the nonzero singular values $\sigma_{r+1},\cdots, \sigma_s$ and the corresponding singular vectors may also be truncated because of their small magnitudes. Consequently, all matrices $V_{m\times r},~ \Sigma_{r\times r},$ and $W_{r\times n}$ are easier to be stored due to their smaller size. This is why POD can be used to optimize data storage and reduce memory requirements.

	\subsection{Incremental proper orthogonal decomposition}\label{incremental}
	The iPOD uses an online projection and QR decomposition to process data one by one and compute the POD (or SVD).  
	The main strength of iPOD is that it does not require pre-storing the entire data matrix, which was proposed by Brand for data in standard Euclidean space $\mathbb{R}^m$ \cite{brand2002} and then extended to data in a general finite-dimensional Hilbert space $\mathbb{R}^m_M$ to deal with functional data \cite{MR3775096,MR4017489,MR3986356}. We here recall the iPOD for data in general finite-dimensional Hilbert space, which primarily summarizes development in literature \cite{brand2002,MR3775096,MR4017489,MR3986356}. Especially, we consider a block-wise data processing \cite{Yw2022} to reduce the loss of orthogonality issue \cite{brand2002,MR3775096,MR4017489,MR3986356} that often happens to the POD basis due to accumulated round-off error from many matrix multiplications. 
	
	%The classical iPOD Algorithms \cite{brand2002,MR3775096,MR4017489,MR3986356} often suffers from loss of orthogonality of the POD basis due to round off error from constantly matrix multiplications, which leads to an inaccurate projection and POD basis update. Thus reorthogonalization has to be applied and is the most computationally expensive step in iPOD. In the following, we summarize from two techniques: matrix separation \cite{BM06} and block-wise data processing \cite{Yw2022}, to tackle the orthogonality issue and obtain accurate POD more effectively. 
	
	To start, we provide a few matrix notations: For a matrix $M$, let  $M_{(:,c:d)}$ denote the submatrix with the columns $c,c+1,\cdots,d$ of $M$, let  $M_{(c:d,:)}$ denote the submatrix with the rows $c,c+1,\cdots,d$ of $M$, let $M_{(a:b,c:d)}$ denote a submatrix that has elements belonging to both rows $a,a+1,\cdots,b$ and columns $c,c+1,\cdots,d$ of $M$.
	
	Let $U$ be a data matrix and $U= V_{m\times l}\Sigma_{l\times l} W_{j\times l}^T$ be the core $M$-weighted SVD and $\vec{u}^{j+1}$ be a new column data available to be processed. We first project $\vec{u}^{j+1}$  onto $V$:
	%\begin{align}\label{proc}
	%VV^*u=V_eV_Q(V_eV_Q)^*u=V_eV_Q(V_eV_Q)^TMu=V_eV_e^TMu=V_eV_e^*u.
	%\end{align}
	%Continuing on (\ref{proc}) we generate a new orthonormal basis $e$ by \cite{BM06}         
	\begin{align}\label{start}
		&\underbrace{VV^*\vec{u}^{j+1}}_{\text{projecting $\vec{u}^{j+1}$ onto $V$}}\Longrightarrow \underbrace{\tilde{e}=\vec{u}^{j+1}-VV^*\vec{u}^{j+1}}_{\text{new orthogonal vector}}\Longrightarrow \underbrace{e=\frac{\tilde{e}}{\|\vec{u}^{j+1}-VV^*\vec{u}^{j+1}\|_{\mathbb{R}_M^m}}}_{\text{normalization}}.
	\end{align}
	Let $p=\|\vec{u}^{j+1}-VV^*\vec{u}^{j+1}\|_{\mathbb{R}_M^m}$. We then decompose the new matrix $\begin{bmatrix}U& \vec{u}^{j+1}\end{bmatrix}$ as:
	\begin{equation}\label{CD}
		\begin{split}
			\begin{bmatrix}U& \vec{u}^{j+1}\end{bmatrix}&= \begin{bmatrix}V\Sigma W^T& \vec{u}^{j+1}\end{bmatrix}=\underbrace{\begin{bmatrix}V&  e\end{bmatrix}  \begin{bmatrix}
					\Sigma W^T& V^*\vec{u}^{j+1}\\0& p
			\end{bmatrix}}_{\text{QR decomposition}}\\
			&=\begin{bmatrix}V&  e\end{bmatrix}  \begin{bmatrix}
				\Sigma& V^*\vec{u}^{j+1}\\0& p
			\end{bmatrix}\begin{bmatrix}
				W^T & 0\\0& 1
			\end{bmatrix}.
		\end{split}
	\end{equation}
	
	Let $\texttt{tol}_p$ be a user-defined threshold that defines an acceptable amount of approximate linear dependence between $\vec{u}^{j+1}$ and $V$. 
	
	{\bf p-truncation:} If $p=\|\vec{u}^{j+1}-VV^*\vec{u}^{j+1}\|_{\mathbb{R}_M^m}< \texttt{tol}_p$, we approximately consider that the new data $\vec{u}^{j+1}$ is linearly dependent on $V$, i.e., $p\approx 0$ or $\vec{u}^{j+1}\approx VV^*\vec{u}^{j+1}$. Assume the consecutive available data $\vec{u}^{j+1}, \vec{u}^{j+2},\cdots,\vec{u}^{j+d}$ are all approximately linearly dependent on $V$. Based on the decomposition in (\ref{CD}), all $e$ and $p$ generated by $\vec{u}^{j+1}, \vec{u}^{j+2},\cdots,\vec{u}^{j+d}$ are zero.  We then have \cite{Yw2022}:
	\begin{align}\label{ptt}
		\begin{bmatrix}U& \vec{u}^{j+1}&\cdots&\vec{u}^{j+d}\end{bmatrix}%&\approx \begin{bmatrix}V\Sigma W^T& VV^*\vec{u}^{j+1}&VV^*\vec{u}^{j+2}&\cdots&VV^*\vec{u}^{j+d}\end{bmatrix}\\
		&\approx V \begin{bmatrix}
			\Sigma& V^*\vec{u}^{j+1}& \cdots& V^*\vec{u}^{j+d}
		\end{bmatrix}\begin{bmatrix}
			W^T & 0\\0& I_d
		\end{bmatrix}.
	\end{align}
	Let $Q=\begin{bmatrix}
		\Sigma& V^*\vec{u}^{j+1}& V^*\vec{u}^{j+2}&\cdots& V^*\vec{u}^{j+d}
	\end{bmatrix}$ and $Q={V}_Q{\Sigma}_Q {W}^T_Q$ be the core standard SVD. We then have
	\begin{equation}\label{2dnew}
		\begin{split}
			\begin{bmatrix}U& \vec{u}^{j+1}&\cdots&\vec{u}^{j+d}\end{bmatrix}&\approx V {V}_Q{\Sigma}_Q {W}^T_Q\begin{bmatrix}
				W^T & 0\\0& I_d \end{bmatrix}\\&=\left(V {V}_Q\right){\Sigma}_Q \left(\begin{bmatrix}
				W & 0\\0& I_d \end{bmatrix}{W}_Q\right)^T.
		\end{split}
	\end{equation}
	It is not difficult to verify that $V {V}_Q$ is $M$-orthonormal, ${\Sigma}_Q $ is ordered diagonal, and $\begin{bmatrix}
		W & 0\\0& I_d \end{bmatrix}{W}_Q$ is standard orthonormal, which tells us that $\begin{bmatrix}U& \vec{u}^{j+1}&\cdots&\vec{u}^{j+d}\end{bmatrix}\approx \left(V {V}_Q\right){\Sigma}_Q \left(\begin{bmatrix}
		W & 0\\0& I_d \end{bmatrix}{W}_Q\right)^T$ is the approximated core $M$-weighted SVD. Also, note that $\begin{bmatrix}
		W & 0\\0& I_d \end{bmatrix}{W}_Q=\begin{bmatrix}
		W & 0\\0& I_d \end{bmatrix}\begin{bmatrix}{W}_{Q(1:l,1:l)}\\{W}_{Q(l+1:l+d,1:l)}\end{bmatrix}=\begin{bmatrix}
		W {W}_{Q(1:l,:)}\\{W}_{Q(l+1:l+d,:)} \end{bmatrix}$. Therefore, we update the POD as:
	\begin{align}\label{PODupdate3}
		V\longleftarrow V{V}_Q,\quad\Sigma\longleftarrow{\Sigma}_Q,\quad W\longleftarrow \begin{bmatrix}
			W {W}_{Q(1:l,:)}\\{W}_{Q(l+1:l+d,:)} \end{bmatrix}.
	\end{align}

	{\bf SV-truncation:} If $p=\|\vec{u}^{j+1}-VV^*\vec{u}^{j+1}\|_{\mathbb{R}_M^m}\geq \texttt{tol}_p$, the new data $\vec{u}^{j+1}$ is linearly independent on $V$. Continuing on (\ref{CD}), we let $\widetilde{Q}=\begin{bmatrix}
		\Sigma& V^*\vec{u}^{j+1}\\0& p
	\end{bmatrix}$ and $\widetilde{Q}=\widetilde{V}_Q\widetilde{\Sigma}_Q \widetilde{W}^T_Q$ be the core standard SVD. Then the matrix $\begin{bmatrix}U& \vec{u}^{j+1}\end{bmatrix}$ is decomposed as:
	\begin{equation}\label{2dnew1}
		\begin{split}
			\begin{bmatrix}U& \vec{u}^{j+1}\end{bmatrix}&=\begin{bmatrix}
				V& e
			\end{bmatrix}\widetilde{V}_Q\widetilde{\Sigma}_Q\widetilde{W}^T_Q\begin{bmatrix}
				W^T & 0\\0& 1
			\end{bmatrix}\\&=\left(\begin{bmatrix}
				V& e
			\end{bmatrix}\widetilde{V}_Q\right)\widetilde{\Sigma}_Q \left(\begin{bmatrix}
				W & 0\\0& 1
			\end{bmatrix}\widetilde{W}_Q\right)^T.
		\end{split}
	\end{equation}
	Again, one can easily check that $\begin{bmatrix}
		V& e
	\end{bmatrix}\widetilde{V}_Q$ is $M$-orthonormal, $\widetilde{\Sigma}_Q$ is ordered diagonal, and $\begin{bmatrix}
		W & 0\\0& 1
	\end{bmatrix}\widetilde{W}_Q$ is standard orthonormal.
	
	Let $\Sigma=\texttt{diag}(\lambda_1,\lambda_2,\cdots,\lambda_l)$ and $\widetilde{\Sigma}_Q=\texttt{diag}(\mu_1,\mu_2,\cdots,\mu_{l+1})$. According to \cite[Lemma 3]{Yw2022}, we have the following relationship among $\{\lambda_i\}$, $\{\mu_i\}$ and $p$:
	\begin{align}\label{lm}
		\mu_{l+1}\leq p,~~
		\mu_{l+1}\leq \lambda_l\leq \mu_{l}\leq \lambda_{l-1}\cdots\mu_2\leq\lambda_1.
	\end{align}
	The inequalities in (\ref{lm}) inform us that the singular value $\mu_{l+1}$ may be small and negligible. Thus, we may truncate  $\mu_{l+1}$ and the corresponding singular vectors. Let $\texttt{tol}_{sv}$ be a user-defined threshold of the singular value  truncation (SV-truncation). If $\mu_{l+1}< \texttt{tol}_{sv}$, we apply SV-truncation and have the POD update:
	\begin{align}\label{PODupdate2}
		V\longleftarrow \begin{bmatrix}V&e\end{bmatrix}\widetilde{V}_{Q(:,1:l)},\quad\Sigma\longleftarrow\widetilde{\Sigma}_{Q(1:l,1:l)},\quad W\longleftarrow \begin{bmatrix}
			W & 0\\0& 1
		\end{bmatrix}\widetilde{W}_{Q(:,1:l)}.
	\end{align}

	{\bf Exact update:} If $p=\|\vec{u}^{j+1}-VV^*\vec{u}^{j+1}\|_{\mathbb{R}_M^m}\geq \texttt{tol}_p$ and $\mu_{l+1}\geq \texttt{tol}_{sv}$, then no truncation is applied. We thus have the exact POD update: 
	\begin{align}\label{PODupdate}
		V\longleftarrow \begin{bmatrix}
			V& e
		\end{bmatrix}\widetilde{V}_Q,\quad\Sigma\longleftarrow\widetilde{\Sigma}_Q,\quad W\longleftarrow\begin{bmatrix}
			W & 0\\0& 1
		\end{bmatrix}\widetilde{W}_Q.
	\end{align}

	The formulas (\ref{PODupdate3}), (\ref{PODupdate2}), and (\ref{PODupdate}) are the foundation for the iPOD Algorithm. We summarize (\ref{start})-(\ref{PODupdate}) in Algorithm \ref{iPOD}.
	
	\begin{algorithm}
		\caption{iPOD Algorithm}\label{iPOD}
		\begin{algorithmic}[1]
			\STATE{\bf Initialization:} $d=0$, $e_p=0$, $e_{sv}=0$, $V_0=1$, $V=\frac{\vec{u}^1}{\|\vec{u}^{1}\|_{\mathbb{R}_M^m}}$,  $\Sigma=\|\vec{u}^1\|_{\mathbb{R}_M^m}$, and $W=1$;
			
			\STATE		{\bf{Input}:} $\vec{u}^{j+1},M$, $V_{m\times l}$, $\Sigma_{l\times l}$, $W_{j\times l}$, $V_0$, $\texttt{tol}_p$, $\texttt{tol}_{sv}$, $\texttt{tol}_{o}$, $d$, $e_p$, $e_{sv}$;~~%\% $\texttt{tol}_{o}$ is an orthogonality tolerance.
			\vspace{0.4mm}
			%		\For{j=1:n}
			\STATE Compute $\vec{b}=V^TM\vec{u}^{j+1},~\tilde e=\vec{u}_{j+1}-V\vec{b}, ~p = \|\tilde e\|_{\mathbb{R}_M^m}$; 
			\vspace{0.4em}
			\IF{$p< \texttt{tol}_p$ $\&$ $ j<n$}
			\STATE $d=d+1$;
			\STATE $B_{(:,d)}= \vec{b}$;
			\STATE $e_p=e_p+p$;\quad \% Data compression error from $p$-truncation
			\ELSE
			\IF{$d> 0$} 
			
			\STATE Compute the core standard SVD of $ Q $: $\begin{bmatrix}
				\Sigma&B
			\end{bmatrix} ={V}_{Q}{\Sigma}_Q{W}_Q$;
			\vspace{0.4em}
			\STATE $V_0={V}_Q$;
			\STATE Update:  $\Sigma={\Sigma}_Q, W=\begin{bmatrix}
				W {W}_{Q(1:l,:)}\\{W}_{Q(l+1:l+d,:)} \end{bmatrix}$;%~~\% Leave the update of $V$ with no $p$-truncation data
			\vspace{0.4em}
			\STATE Update: $\vec{b}={V}_Q^T\vec{b}$; %\quad \% Update h with the latest V for the data without $p$-truncation
			\ENDIF
			\STATE Set $e=\frac{\tilde e}{p};$ 
			\WHILE{$V(:,1)^TMe>\texttt{tol}_o$} 
			\STATE $e=e-VV^TMe$;
			\STATE $e=\frac{e}{\| e\|_{\mathbb{R}_M^m}}$;
			\ENDWHILE
			\STATE Set $V=\begin{bmatrix}V&e\end{bmatrix}, \widetilde{Q}=\begin{bmatrix}
				\Sigma&\vec{b}\\0&p
			\end{bmatrix}$;
			\STATE Compute the core standard SVD  of $\widetilde{Q}:~ =\widetilde{V}_{Q}\widetilde{\Sigma}_Q\widetilde{W}_Q$;
			\vspace{0.4em}
			
			% \State $[a,b]=size(V_Q)$;
			\STATE    Set   $V_1=\begin{bmatrix}
				V_0&0\\	0& 1			\end{bmatrix}$;
			%   \vspace{0.4em}
			%  \State $l$=\texttt{length}($\widetilde{\Sigma}_Q$);
			\IF{$\widetilde{\Sigma}_{Q(l+1,l+1)}>\texttt{tol}_{sv}$}
			\STATE Update: $V=\begin{bmatrix}
				V&e				\end{bmatrix}V_1\widetilde V_{Q}, \Sigma=\widetilde{\Sigma}_Q, W=\begin{bmatrix}
				W & 0\\0& 1
			\end{bmatrix}\widetilde{W}_Q,V_0=I_{l+1}$;
			%	\vspace{0.4em}
			\ELSE 
			\STATE  {Update: $\widetilde{V}_{Q}=V_1\widetilde{V}_{Q}, V=\begin{bmatrix}
					V&e\end{bmatrix}\widetilde{V}_{Q(:,1:l)},  
				\Sigma=\widetilde{\Sigma}_{Q(1:l,1:l)}$, \\\qquad\qquad $W=\begin{bmatrix}
					W & 0\\0& 1
				\end{bmatrix}\widetilde{W}_{Q(:,1:l)}, V_0=I_{l}$};
			\STATE{ $e_{sv}=e_{sv}+\widetilde{\Sigma}_{Q(l+1,l+1)}$;\quad\% Data compression error from SV-truncation}
			\ENDIF
			\STATE $B=[~ ]$, $d=0$;
			\ENDIF
			%		\EndFor
			
			\STATE		{\bf{Return}:} $V=V, \Sigma=\Sigma, W=W, d, e_p, e_{sv}, V_0$;
		\end{algorithmic}
	\end{algorithm}
	
	\begin{remark}
		In Algorithm \ref{iPOD}, the projection to generate $``e"$ may contain round-off error so that the matrix $\begin{bmatrix}V& e\end{bmatrix}$ is not $M$-orthonormal.  We simply check and fix this issue by projecting multiple times \cite{Lbj2012}; see lines $16-19$ in Algorithm \ref{iPOD}.
		
		In line $7$ and line $27$, we propose an incremental way to calculate the data compression error from $p$-truncation and SV-truncation. We prove in Section \ref{ErrorAnalysis} that this is a sharp estimate for the data compression error. 
		
		Also, note that this block-wise data treatment in (\ref{ptt}) does not cause any storage issue even for large $d$ since the matrix $ \begin{bmatrix}
			\Sigma& V^*\vec{u}^{j+1}& V^*\vec{u}^{j+2}&\cdots& V^*\vec{u}^{j+d}
		\end{bmatrix}$ is normally a wide-short small size matrix. 
	\end{remark}
	
	Algorithm \ref{iPOD} differs from classical iPOD algorithms \cite{brand2002,MR3775096,MR4017489,MR3986356} primarily in lines $4-14$. The classical iPOD Algorithms update the POD basis $V$ by matrix multiplication as every new column data comes in and will encounter loss of orthogonality on $V$ quickly due to accumulated round-off error. Hence, the Gram-Schmidt procedure must be applied often to reorthogonalize $V$, which is the most computationally expensive step in classical iPOD. On the other hand, the block-wise processing of $p$-truncation data from \cite{Yw2022} in Algorithm \ref{iPOD} only involves one-time matrix multiplication to update the $V$ and $W$ for $d$ columns data, which largely reduces the source of round-off error, preserves the orthogonality of $V$, and allows us to obtain accurate POD with less cost. %This is the main advantage of Algorithm \ref{iPOD} outperforming the classical iPOD Algorithms. 

	\section{iPOD data compression for time-dependent PDE-constrained optimization}\label{Inexactgradientmethod}
	Next, we present the iPOD data compression for the optimization problem with time-dependent PDE constraint, by using a data assimilation problem for the detailed demonstration. From now on, let $V$ and $H$ denote separable Hilbert spaces. Let $V\subset H\subset V'$ be a triplet over a domain $\Omega$, where $V'$ is the dual space of $V$. The function spaces $L^2(0,T;X)$ and $L^\infty(0,T;X)$ are given by $L^2(0,T;X) := \{v \,| \, v(t)\in X ~\text{for a.e.\ }t\in (0,T)\text{~and } \int_0^T\|v\|_X^2dt<\infty \}$ and $L^{\infty}(0,T;X) := \{v \, | \, v(t)\in X \text{~and } \|v(t) \|_X<\infty ~\text{for a.e.\  }t\in (0,T)\}$. A time-dependent PDE is generally written as (after suppressing the boundary conditions):
	\begin{equation}\label{Cconstraint}
		\left\{\begin{aligned}
			&\frac{\partial u}{\partial t}+Au+F(u)=f\quad \text{in}~V'\\%L^2(0,T;V'),\\
			&u(\cdot,0)=u_{0}\quad \text{in~} H.
		\end{aligned}\right.
	\end{equation}
	Here, $A: V\mapsto V'$ and $F: V\mapsto V'$ are generic linear and nonlinear operators.
	
	Time-dependent PDE-constrained optimization problems have many applications, and each application may involve different formulations with respect to either the objective function or the constraint. Nevertheless, the use of iPOD will all happen in the same step that compresses the forward solution data to save storage and make the gradient calculation possible. In this paper, we concretely consider a data assimilation problem, which is a data-driven solution discovery problem involving initial identification, as a typical example of the PDE-constrained optimization for presentation. The iPOD data compression and the analysis can certainly be extended to other types of optimization problems with time-dependent PDE constraints, which is interesting future work. 
	
	In this paper, we focus on the iPOD data compression for the following data assimilation problem. Assume the initial condition $u_{0}$ in equation (\ref{Cconstraint}) is unknown. We utilize observation data $\hat{u}\in L^2(0,T;\mathscr{H})$ and the PDE to attempt to identify $u_{0}$ and the corresponding solution $u$ by considering the optimization problem \cite{MR0271512,Manzoni21,Tr2010}
	\begin{eqnarray}\label{Cfunctional}
		\min_{u_{0}\in \mathscr{U}}J_{}(u_{0})=\frac{1}{2}\int_0^T\|\hat u^{}-\mathcal{C}u_{}^{}\|^{2}_{\mathscr{H}}dt+\frac{\gamma}{2}\|u_{0}\|^{2}_{\mathscr{U}}\quad\text{subject to } (\ref{Cconstraint}).
	\end{eqnarray}
	Here, $\mathcal{C}: V\mapsto \mathscr{H}$ is an observation operator, and $\mathscr{H}$ and $\mathscr{U}$ are Hilbert spaces. %Problem (\ref{Cfunctional}) is usually called a data driven solution discovery problem, which will be used as a typical example of the PDE-constrained optimization for presentation in this paper.
	
	At the discrete level, the problem (\ref{Cfunctional}) can be approximated as: 
	\begin{eqnarray}\label{discretecostChapter3}
		\min_{u_{0,h}\in \mathscr{U}_h}J_{}(u_{0,h})=\frac{1}{2}\tau \sum_{j=1}^{n}\|\hat u^{n}-\mathcal{C}^ju_{h}^{j}\|^{2}_{\mathscr{H}}+\frac{\gamma}{2}\|u_{0,h}\|^{2}_{\mathscr{U}}.
	\end{eqnarray}
	subject to 
	\begin{equation}\label{discreteoptChapter3}
		\left\{\begin{aligned}
			&\frac{u^{j+1}_h-u^{j}_h}{\tau}+A u^{j+1}_h+F(u^{j+1}_h)=f^{j+1}\quad \text{in}~V_h',\\
			&u^0_h=u_{0,h}\quad \text{in}~H,
		\end{aligned}\right.
	\end{equation}
	for all $j=0,1,\cdots,n-1$. Here, $\mathcal{C}^j: V\mapsto \mathscr{H}$ is an observation operator, and $\mathscr{U}_h$ is a discrete subspace of $\mathscr{U}$. Note that in (\ref{discreteoptChapter3}) backward Euler is used for the temporal discretization of the PDE constraint (\ref{Cconstraint}). The operators $A$ and $F$ involve the spatial discretization, and we still denoted them by $A$ and $F$. The index $j=0,1,2,3,...,{n-1}$ represents the time moment, $\tau$ is the time step size, and 
	$h$ is the spatial mesh size of the uitilized numerical methods, such as the finite element (FE) methods, finite difference methods, or finite volume methods. In the rest of the paper, we consider FE discretization and the FE space is accordingly denoted by $V_h=\mathrm{span}\{\phi_1, \phi_2, \cdots, \phi_m\}$, where  $\{\phi_i\}$ is a generic FE basis. For $\forall v_h\in V_h$, we have $v_h=\sum_{i=1}^m a_i\phi_i$ and $\vec{v}_h=\begin{bmatrix}a_1&a_2&\cdots&a_m\end{bmatrix}^T $ is the vector representation of $v_h$. 
	
	Again, gradient methods, such as steepest descent, conjugate gradient, and quasi-Newton methods, are common mechanisms to solve the problem (\ref{discretecostChapter3})-(\ref{discreteoptChapter3}), which all require calculating the gradient of (\ref{discretecostChapter3}) at $u_{0,h}^{}$ by $j=0,1,2,\cdots,n-1$ \cite{Max1996}%of the objective function (\ref{discretecostChapter3}), see e.g., (\ref{eqnfor})-(\ref{eqngradient}):
	\begin{equation}
		\label{eqnfor}
		\left\{\begin{aligned}
			&\frac{ {u}_{h}^{j+1}- {u}_{h}^{j}}{\tau}+ A{u}_{h}^{j+1}+F(u^{j+1}_h)={f}^{n+1},\\
			& {u}_{h}^{0}= u_{0,h}^{},
		\end{aligned}\right.
	\end{equation}
	\begin{equation}
		\label{eqnback}
		\left\{\begin{aligned}
			&-\frac{ {u}_{h}^{*j+1}-{u}_{h}^{*j}}{\tau}+A^*{u}_{h}^{*j}+\left(F'(u^{j+1}_h)\right)^*{u}_{h}^{*j}=\mathcal{C}^{j+1*}\Lambda(\hat {u}^{j+1}-\mathcal{C}^{j+1}{u}^{j+1}_h),\\
			&{u}_h^{*n}= 0,
		\end{aligned}\right.
	\end{equation}
	\begin{align}\label{gradient} 
		\nabla J({u}^{}_{0,h})=-{u}^{*0}_{h}|_{{\mathscr{U}_h}}+\gamma  {u}^{}_{0,h}, 
	\end{align}
	where $A^*$ and $\mathcal{C}^{j+1*}$ are the adjoint operator of $A$ and $\mathcal{C}^{j+1}$, $\Lambda$ is the canonical isomorphism of $\mathscr{H}$ and its dual space, $\left(F'(u^{j+1}_h)\right)^*{u}_{h}^{*j}$ denotes the adjoint of the derivative of $F$ at $u^{j+1}_h$ acting on ${u}_{h}^{*j}$, and ${u}^{*0}_{h}|_{{\mathscr{U}_h}}$ satisfies $({u}^{*0}_{h},v)_H=({u}^{*0}_{h}|_{{\mathscr{U}}},v_h)_{\mathscr{U}}~~\forall v_h\in {{\mathscr{U}_h}}$.
	
	In the rest of the paper, we use the steepest descent method as a representative of gradient-based methods to illustrate using the iPOD data compression approach. The steepest descent method states that:  
	\begin{align}\label{eqngradient}
		{u}^{(i+1)}_{0,h}={u}^{(i)}_{0,h}-\kappa\nabla J({u}^{(i)}_{0,h}),
	\end{align}
	where $i$ is the $i^{th}$ gradient iteration and $\kappa$ is a positive constant called the descent step size or learning rate. 
	
	%\vspace{2mm}
	We summarize the steepest descent method in Algorithm \ref{SD}.%for solving problem (\ref{Cfunctional})-(\ref{discreteoptChapter3}) 
	%as follows:
	\begin{algorithm}%[H]
		\caption{Steepest Descent Algorithm}\label{SD}
		\begin{algorithmic}
			\STATE		{\bf{Input}:} $u_{0,h}^{(0)}$, and $\mathrm{\texttt{tol}}_{sd}$; \quad \% $\mathrm{\texttt{tol}}_{sd}$ denotes the termination tolerance
			%	\State Initialization: Compute $V_1=\frac{c_1}{\|c_1\|_M}$, $\Sigma_1=\|c_1\|_M$, and $W=1$;
			\STATE	Set $\mathrm{error}=1$ and $i=0$. 
			\WHILE{$\mathrm{error}>\mathrm{\texttt{tol}}_{sd}$}
			\STATE Forward phase: solve equation (\ref{eqnfor}) forward with initial $u_{0,h}^{(i)}$ to obtain $\{u^{j(i)}_h\}_{j=1}^n$;
			\STATE Backward phase: solve equation (\ref{eqnback}) backward with $\{u^{j(i)}_h\}_{j=1}^n$ to obtain $u^{*0(i)}_h$;
			\STATE Update: ${u}^{(i+1)}_{0,h}={u}^{(i)}_{0,h}-\kappa\nabla J({u}^{(i)}_{0,h})$;
			\STATE Set $i=i+1$ and $ \mathrm{error} = \kappa\|\nabla J({u}^{(i)}_{0,h})\|_H$;
			\ENDWHILE
			
			\STATE		{\bf{Return}:} $\{{u}^{j(i)}\}_{j=0}^n$;
		\end{algorithmic}
	\end{algorithm}
	
	In Algorithm \ref{SD}, the storage of the solution data $\{u^{j(i)}_h\}_{j=1}^n$ (or data matrix $\begin{bmatrix}\vec{u}^{1(i)}_h&\vec{u}^{2(i)}_h&\cdots,&\vec{u}^{n(i)}_h\end{bmatrix}$ under a prescribed FE basis) at each iteration is problematic when the forward PDE simulation is large-scale. Hence, we apply iPOD at the forward phase to compress the data $\begin{bmatrix}\vec{u}^{1(i)}_h&\vec{u}^{2(i)}_h&\cdots,&\vec{u}^{n(i)}_h\end{bmatrix}$ to a smaller size,  and then decompress it step by step as needed for solving equation (\ref{eqnback}) to obtain $u^{*0(i)}_h$ and the gradient $\nabla J({u}^{(i)}_{0,h})= {u}^{*0}_{h}|_{{\mathscr{U}_h}}-\gamma {u}^{(i)}_{0,h}$. This is an inexact gradient method because there is a gradient error at each iteration from solving the backward equation with approximated data of $\{{u}^{j(i)}\}_{j=1}^n$.
	
	We present this inexact gradient method in Algorithm \ref{iPODSD}.
	\begin{algorithm}
		\caption{Inexact Gradient Algorithm}\label{iPODSD}
		\begin{algorithmic}
			\STATE 		{\bf{Input}:} $u_{0,h}^{(0)}$, and $\mathrm{\texttt{tol}}_{sd}$;
			%	\State Initialization: Compute $V_1=\frac{c_1}{\|c_1\|_M}$, $\Sigma_1=\|c_1\|_M$, and $W=1$;
			\STATE	Set $ \mathrm{error} =1$ and $i=0$. 
			\WHILE{$ \mathrm{error} >\mathrm{\texttt{tol}}_{sd}$}
			\STATE Forward phase: implement iPOD Algorithm \ref{iPOD} step by step as data $\{\vec{u}^{j(i)}_h\},~ j=1,2,\ldots, n$, becomes available from forward solving (\ref{eqnfor}) with initial $u_{0,h}^{(i)}$ and obtain the compressed data matrix $V\Sigma W^T\approx\begin{bmatrix}\vec{u}^{1(i)}_h&\vec{u}^{2(i)}_h&\cdots,&\vec{u}^{n(i)}_h\end{bmatrix} $;
			\STATE Backward phase: use $V, \Sigma,$ and $W$ to reconstruct $\{{u}^{j(i)}_h\},~ j=n,n-1,\cdots, 1$ step by step to solve the backward equation (\ref{eqnback}) and obtain $u^{*0(i)}_h$;
			\STATE Update: ${u}^{(i+1)}_{0,h}={u}^{(i)}_{0,h}-\kappa\nabla J({u}^{(i)}_{0,h})$;
			\STATE Set $i=i+1$ and $ \mathrm{error} = \kappa\|\nabla J({u}^{(i)}_{0,h})\|_H$;
			\ENDWHILE
			
			\STATE		{\bf{Return}:} $\{{u}^{j(i)}\}_{j=0}^n$;
		\end{algorithmic}
	\end{algorithm}
	
	\begin{remark}
		Note that, in the Forward phase of Algorithm \ref{iPODSD}, we have not provided the weight matrix $M$ yet for iPOD; this will be discussed and provided rigorously through analysis in the following sections. 
		
		In addition, in the Backward phase, to reconstruct ${u}^{j(i)}_h$ from $V, \Sigma,$ and $W$, we first get the vector by $\vec{u}^{j(i)}_h=\frac{1}{\tau}\left(V\Sigma\right) W_{(j,:)}^T$ and then obtain ${u}^{j(i)}_h$ by expanding vector $\vec{u}^{j(i)}_h$ with the given FE basis $\{\phi_i\}_{i=1}^m$.
	\end{remark}
	\vspace{1mm}
	
	This inexact gradient method is memory-friendly since we only need to store small size matrices $V$, $\Sigma$, and $W$ instead of $\{\vec{u}^{j(i)}_h\}_{j=1}^n$. However, the gradient $\nabla J({u}^{(i)}_{0,h})$ at each iteration has error due to the data compression. Hence, before applying Algorithm \ref{iPODSD}, we need to answer a few critical questions: Does the inexact gradient method still converge to the optimal solution? If so, how does the convergence behave? We will resolve these concerns with a complete  analysis in next section.

	\section{Analysis for the inexact gradient method}\label{SDA}
	The analysis consists of three steps: iPOD data compression error estimate ${\bf \longrightarrow}$ gradient error analysis $\longrightarrow$ convergence of the inexact gradient method. Hilbert-Schmidt operators will be vital in estimating the iPOD data compression error. Using a numerical PDE analysis, we analyze and bound the gradient error by the data compression error. Finally, the gradient error bound and a convex optimization analysis allow us to prove convergence of the inexact gradient method. % while the gradient error is appropriately bounded. 
	This analysis framework is generally applicable for any PDE-constrained optimization. The details may differ based on the different PDE problems or gradient methods considered. 
	
	\iffalse
	\textcolor{red}{no new modification for this part so far.}
	\textcolor{blue}{Section 4 NEW NOTE: If we want, I think we can extend the iPOD error analysis to deal with a second norm. Specifically, suppose we are incrementally computing the SVD with respect to the weight $ M $, but we want to know the data compression error using a different weight $ E $. I think this is doable without much work! We will need to compute $ p_E $, which is exactly the formula for $ p $, but with the new weight $ E $. The SVD truncation error with the different norm is easy to find; I have done this in other papers. This would allow us to handle the Navier-Stokes case with iPOD computed with the discrete $ L^2 $ inner product, but of course we need to know the data compression error in the (stronger) discrete $ H^1 $ inner product. We simply need to decide if this is worth the extra effort at this stage; we need to finish this paper! I do think if we add this analysis, then we might be able to submit this paper to SINUM...}
	\fi
	
	\subsection{iPOD data compression error estimate}\label{ErrorAnalysis}
	The authors of \cite{MR4017489} provide an iPOD error estimate between the exact and compressed data with respect to the infinity matrix norm. We obtain a natural improvement using on the Hilbert-Schmidt norm, since the relevant quantities in POD theory are measured by the Hilbert-Schmidt norm, see (\ref{H2}).

	%As known that a finite rank Hilbert-Schmidt operator $S$ can be represented by a matrix $C$. In this sense, we use the notation $\|C\|_\mathrm{HS}=\|S\|_\mathrm{HS}$.

	% of $T:\mathbb{R}^m\to \mathbb{R}^n_M$.
	\begin{lemma}\label{Innorm}
		Let $U^{j}: \mathbb{R}^j\mapsto \mathbb{R}^m_M $ be a given data matrix, $\vec{u}^{j+1}\in \mathbb{R}_M^m$ be a new column data, and $U^{j+1}$ be $\begin{bmatrix}U^{j}\quad \vec{u}^{j+1}\end{bmatrix}:\mathbb{R}^{j+1}\mapsto \mathbb{R}^n_M $. Then we have $\|U^{j+1}\|_\mathrm{HS}^2=\|U^{j}\|_\mathrm{HS}^2+\|\vec{u}^{j+1}\|_{\mathbb{R}^m_M}^2$.
	\end{lemma}
	\begin{proof}
		Simply using Lemma \ref{HSnorm}, we derive
		\begin{align*}	&\|U^{j+1}\|_\mathrm{HS}^2=\sum_{i=1}^{j+1}\|\vec{u}^{i}\|_{\mathbb{R}^m_M}^2=\sum_{i=1}^{j}\|\vec{u}^{i}\|_{\mathbb{R}^m_M}^2+\|\vec{u}^{j+1}\|_{\mathbb{R}^m_M}^2=\|U^{j}\|_\mathrm{HS}^2+\|\vec{u}^{j+1}\|_{\mathbb{R}^m_M}^2.
		\end{align*}  
	\end{proof}
	
	As presented in Section \ref{incremental} for the iPOD Algorithm, the data compression errors are from two parts: the $p$-truncation and SV-truncation. Based on Lemma \ref{Innorm}, we estimate these two errors with the following lemmas. 
	\vspace{0.5mm}
	
	{\bf p-truncation error:} Let $U^{j}:\mathbb{R}^j\mapsto \mathbb{R}^m_M $ be a given data matrix and $U^j=V^j\Sigma^j\left(W^{j}\right)^T$ be the core M-weighted SVD. Once a new data $\vec{u}^{j+1}$ is available to be processed, if $p^{j+1}=\|\vec{u}^{j+1}-V^jV^{j*}\vec{u}^{j+1}\|_{\mathbb{R}_M^m}$ is smaller than $\texttt{tol}_p$, then iPOD will execute a $p$-truncation, i.e., $\vec{u}^{j+1}\approx V^jV^{j*}\vec{u}^{j+1}$, and update the SVD based on matrix $\begin{bmatrix}
		U^j& V^jV^{j*}\vec{u}^{j+1}
	\end{bmatrix}$ instead of $\begin{bmatrix}
		U^j& \vec{u}^{j+1}
	\end{bmatrix}$. %The p-truncation error is evaluated in the following result. 
	
	%\vspace{1mm}
	\begin{lemma}\label{pt}
		Let $U^{j}:\mathbb{R}^j\mapsto \mathbb{R}^m_M $ and  $U^j=V^j\Sigma^j\left(W^{j}\right)^T$ be the core $M$-weighted SVD. Let  $\vec{u}^{j+1}$ be a new column data to be processed, $p^{j+1}=\|\vec{u}^{j+1}-V^jV^{j*}\vec{u}^{j+1}\|_{\mathbb{R}_M^m}$, and  $\widehat{U}^{j+1}=\begin{bmatrix}
			U^j& V^jV^{j*}\vec{u}^{j+1}
		\end{bmatrix}$.  We have \begin{align}
			\|U^{j+1}-\widehat{U}^{j+1}\|_\mathrm{HS}=p^{j+1}=\|\vec{u}^{j+1}-V^jV^{j*}\vec{u}^{j+1}\|_{\mathbb{R}^m_M}.
		\end{align}
	\end{lemma}
	\begin{proof}
		By basic matrix operations and Lemma \ref{Innorm}, we have
		\begin{align*}
			\|U^{j+1}-\widehat{U}^{j+1}\|_\mathrm{HS}^2&=\|\begin{bmatrix}U^j& \vec{u}^{j+1}\end{bmatrix}-\begin{bmatrix}U^j&V^{j}V^{j*}\vec{u}^{j+1}\end{bmatrix}\|_\mathrm{HS}^2\\
			&=\|\begin{bmatrix}0& \vec{u}^{j+1}-V^{j}V^{j*}\vec{u}^{j+1}\end{bmatrix}\|_\mathrm{HS}^2\\
			&=\|\vec{u}^{j+1}-V^{j}V^{j*}\vec{u}^{j+1}\|_{\mathbb{R}^m_M}^2,
		\end{align*} 
		which completes the proof.
	\end{proof}

	{\bf SV-truncation error:} The second truncation applied to iPOD is the SV-truncation: %the SVD includes negligible 
	If the smallest singular value of a data matrix is less than the given threshold $\texttt{tol}_{sv}$, we truncate it and the corresponding singular vector. The following Lemma estimates this truncation.

	\begin{lemma}\label{last}
		Let $U^{j}:\mathbb{R}^j\mapsto \mathbb{R}^m_M $, $\vec{u}^{j+1}$ be a new column data to be processed, and
		$\widetilde{U}^{j+1}$ be the approximated data matrix of ${U}^{j+1}=\begin{bmatrix}
			U^j& \vec{u}^{j+1}
		\end{bmatrix}$ after applying SV-truncation. If $\sigma_l^{j+1}$ is the smallest singular value of  ${U}^{j+1}$ that is truncated, then $$ \|{U}^{j+1}-\widetilde{U}^{j+1}\|_\mathrm{HS}^2=(\sigma_l^{j+1})^2.$$  
	\end{lemma}
	
	Lemma \ref{last} is a direct implication of Lemma \ref{SVDC}. 
	
	Based on Lemma \ref{pt}, Lemma \ref{last}, and the lemmas given in Section \ref{Background}, we provide an incremental way to estimate the error between the exact and compressed data. 
	\begin{theorem}\label{thm0}
		Let $\widetilde{U}^{j}$ and $\widetilde{U}^{j+1}$ be the iPOD approximated data matrices of  the exact data matrices ${U}^{j}:\mathbb{R}^j\mapsto \mathbb{R}^m_M $ and ${U}^{j+1}=\begin{bmatrix}{U}^{j}&\vec{u}^{j+1}\end{bmatrix}:\mathbb{R}^{j+1}\mapsto \mathbb{R}^m_M $, respectively. Let $\widetilde{U}^{j}=V^j\Sigma^j\left(W^{j}\right)^T$ be the core $M$-weighted SVD. 
		If \begin{align*}
			\|U^{j}-\widetilde{U}^j\|_\mathrm{HS}\leq \epsilon^j~ \text{and~  }p^{j+1}=\|\vec{u}^{j+1}-V^jV^{j*}\vec{u}^{j+1}\|_{\mathbb{R}^m_M},
		\end{align*}
		then 
		\begin{align}
			\|U^{j+1}-\widetilde{U}^{j+1}\|_\mathrm{HS}\leq \epsilon^{j+1},
		\end{align}
		where 
		\begin{align}\label{errorestimate}
			\epsilon^{j+1}& \leq
			\begin{cases}
				\epsilon^{j} & \text{if no truncation is applied},\\
				\epsilon^{j}+p^{j+1}& \text{if { p}-truncation is applied},\\
				\epsilon^{j}+\sigma_{l}^{j+1}& \text{if {SV}-truncation is applied},
			\end{cases}
		\end{align}
		and $\{\sigma_i^{j+1}\}_{i=1}^l$ are the singular values of the matrix $\begin{bmatrix}\widetilde{U}^{j}&\vec{u}^{j+1}\end{bmatrix}$.
	\end{theorem}
	\begin{proof} 
		Using the triangle inequality and Lemma \ref{pt}, we deduce
		\begin{equation}\label{bd1}
			\begin{split}
				&\|{U}^{j+1}-\widetilde{U}^{j+1}\|_\mathrm{HS}=\|U^{j+1}-[\widetilde{U}^{j}\quad \vec{u}^{j+1}]+[\widetilde{U}^{j}\quad \vec{u}^{j+1}]-\widetilde{U}^{j+1}\|_\mathrm{HS}\\
				&\leq \|U^{j+1}-[\widetilde{U}^{j}\quad \vec{u}^{j+1}]\|_\mathrm{HS}+\|[\widetilde{U}^{j}\quad \vec{u}^{j+1}]-\widetilde{U}^{j+1}\|_\mathrm{HS}\\
				&\leq \|[U^{j}-\widetilde{U}^{j} \quad 0]\|_\mathrm{HS}+\|[\widetilde{U}^{j}\quad \vec{u}^{j+1}]-\widetilde{U}^{j+1}\|_\mathrm{HS}\\
				&\leq \|U^{j}-\widetilde{U}^{j}\|_\mathrm{HS}+\|[\widetilde{U}^{j}\quad \vec{u}^{j+1}]-\widetilde{U}^{j+1}\|_\mathrm{HS}.
			\end{split}
		\end{equation}
		According to the inequality (\ref{lm}), each data increment only applies at most one truncation, either p-truncation, SV-truncation, or no truncation. If no truncation is applied, the core $M$-weighted SVD is updated exactly and the first inequality of (\ref{errorestimate}) holds. If the p-truncation is applied,  combining (\ref{bd1}) and Lemma \ref{pt} gives us 
		\begin{align}
			\|{U}^{j+1}-\widetilde{U}^{j+1}\|_\mathrm{HS}\leq \|U^{j}-\widetilde{U}^{j}\|_\mathrm{HS}+\|\vec{u}^{j+1}-V^jV^{j*}\vec{u}^{j+1}\|_{\mathbb{R}^m_M}= e^j+p^{j+1}.
		\end{align}
		If the SV-truncation is applied, combining (\ref{bd1}) and Lemma \ref{last} gives us 
		\begin{align}
			\|{U}^{j+1}-\widetilde{U}^{j+1}\|_\mathrm{HS}\leq \|U^{j}-\widetilde{U}^{j}\|_\mathrm{HS}+\sigma_l^{j+1}\leq e^j+\sigma_l^{j+1}.
		\end{align}
		This completes the proof. 
	\end{proof}
	\begin{remark}
		In addition, for one who is interested in the ratio $r^{j+1}$ between the first $j+1$ column approximated data and the exact data, with Lemma \ref{SVDC} and Lemma \ref{HSnorm}, we can calculate it by
		\begin{align}\label{ratio}
			r^{j+1}=\sqrt{\frac{\|\widetilde{U}^{j+1}\|^2_\mathrm{HS}}{\|{U}^{j+1}\|^2_\mathrm{HS}}}=\sqrt{\frac{\|\widetilde{U}^{j+1}\|^2_\mathrm{HS}}{\sum_{i=1}^{j+1}\|\vec{u}^i\|_{\mathbb{R}^m_M}^2}}=\sqrt{\frac{\sum_{\{i|\sigma_i^{j+1}\geq \texttt{tol}_{sv}\}}(\sigma_i^{j+1})^2}{\sum_{i=1}^{j+1}\|\vec{u}^i\|_{\mathbb{R}^n_M}^2}}.
		\end{align}
		Note that the denominator of the ratio can be computed incrementally so that we do not need to store any data. 
	\end{remark}
	
	Theorem \ref{thm0} provides us a theoretical approach to set the $p$-truncation threshold $\texttt{tol}_p$ and the SV-truncation threshold $\texttt{tol}_{sv}$ in the iPOD Algorithm to control the overall data compression error and achieve the desired data approximation accuracy, which additionally implies that the block size of $p$-truncation data ``d'' and the number of POD basis vectors in Algorithm \ref{iPOD} are not necessarily pre-defined; they are all automatically determined by the data and the thresholds $\texttt{tol}_p$ and $\texttt{tol}_{sv}$.
	
	\subsection{Gradient error analysis}
	Recall that the gradient error arises when the compressed data of $\{u_{h}^n\}_{j=1}^n$ is used to solve the backward adjoint equation (\ref{eqnback}). To estimate the gradient error, we need to investigate how the data compression error evolves in the backward equation (\ref{eqnback}). 
	In the following, we analyze such error evolution for example linear and nonlinear PDE constraints, respectively. 
	
	We provide a few notations: Let $W^{k,p}(\Omega)$ denote the standard Sobolev spaces, let $H^1(\Omega)=W^{1,2}(\Omega)$ and $L^2(\Omega)=W^{0,2}(\Omega)$, and let $H^1_0(\Omega)=\{v\in H^1(\Omega): v=0~\text{on}~\partial \Omega\}$.
	For presentation convenience, we here consider $V=H^1_0(\Omega)$, $V'=H^{-1}(\Omega)$, $H=L^2(\Omega)$, $\mathscr{H}=L^2(\Omega)$,  $\mathscr{U}=L^2(\Omega)$, and the observation operator $\mathcal{C}$ and $\mathcal{C}^j$ to be the identity operator in problem (\ref{discretecostChapter3})-(\ref{discreteoptChapter3}). Then the isomorphism $\Lambda$ in (\ref{eqnback}) is an identity operator, and ${u}^{*0}_{h}|_{{\mathscr{U}_h}}$ in (\ref{gradient}) is simply ${u}^{*0}_{h}$. 
	
	We also recall the Poincar\'e inequality: 
	$
	\|v\|_{L^2(\Omega)}\leq C_P\|v\|_{H^1(\Omega)} $  for all $ v\in H^1(\Omega) $, where $C_P$ is a positive constant depending on $\Omega$.
	
	\subsubsection{Gradient error analysis for a linear PDE constraint}\label{LGD}
	We first consider a linear PDE constraint ($F(u)=0$ in (\ref{discreteoptChapter3})) to illustrate how the iPOD error impacts the gradient. Again, let $V_h\subset H^1(\Omega)$ denote the finite element space for the considered linear PDE.
	
	\begin{theorem}\label{thm1}
		Assume $ F = 0 $ and the operator $A^*$ in equation (\ref{eqnback}) satisfies: $\left\langle A^*v,v\right\rangle+C_1\|v\|_{L^2(\Omega)}^2\geq C_2\|v\|_{H^1(\Omega)}^2$  $\forall v\in H^1(\Omega)$, where $C_1 \geq 0$ and $C_2 > 0$ are constants. Let $\{{u}_{h}^{j}\}$ be solutions of equation (\ref{eqnfor}) with a given inital condition $u_{0,h}\in L^2(\Omega)$ and let $\{\vec{u}_{h}^{j}\}$	be a vector representation of $\{{u}_{h}^{j}\}$.
		%Let $\{{u}_{h}^{j}\}_{j=1}^n$ be solutions of equation (\ref{eqnfor}) with a given initial condition $u_{0,h}\in L^2(\Omega)$,
		Define the POD operator $U=\begin{bmatrix}
			\tau^{\frac{1}{2}}\vec{u}_{h}^{1}&\tau^{\frac{1}{2}}\vec{u}_{h}^{2}&...&\tau^{\frac{1}{2}}\vec{u}_{h}^{n}  
		\end{bmatrix}: \mathbb{R}^{n}\mapsto \mathbb{R}^{m}_M$, where the matrix $M$ is induced by the $L^2(\Omega)$ norm. Let $\widetilde{U}=\begin{bmatrix}
			\tau^{\frac{1}{2}}\vec{u}_{h,r}^{1}&\tau^{\frac{1}{2}}\vec{u}_{h,r}^{2}&...&\tau^{\frac{1}{2}}\vec{u}_{h,r}^{n}  
		\end{bmatrix}$  be the approximated data matrix of $U$ by applying the iPOD with respect to the weight matrix $ M $, and $\epsilon$ be an error satisfying $\|U-\widetilde{U}\|_\mathrm{HS}\leq \epsilon$.  Denote $C_3=\frac{2C_1}{1-2C_1\tau}$ and $T=n\tau$.  If $1-2C_1\tau>0$, then the gradient error $\xi$ is bounded by
		\begin{align}
			\|\xi\|_{L^2(\Omega)}\leq C_P\sqrt{\frac{e^{C_3T}}{C_2}} \, \epsilon.
		\end{align}
	\end{theorem}
	
	\begin{proof}
		This is a proof similar to a standard PDE stability analysis with respect to the source term; see Appendix \ref{ThmHeat} for details. 
	\end{proof}
	\begin{remark}
		For parabolic and hyperbolic PDEs, it is typical to have $\left\langle Av,v\right\rangle+C_1\|v\|_{L^2(\Omega)}^2\geq C_2\|v\|_{H^1(\Omega)}^2$  $\forall v\in H^1(\Omega)$, for well-posedness. This property is preserved for $A^*$ because of the linearity of $A$, i.e., $\left\langle A^*v,v\right\rangle+C_1\|v\|_{L^2(\Omega)}^2\geq C_2\|v\|_{H^1(\Omega)}^2$  $\forall v\in H^1(\Omega)$.
	\end{remark}
	\begin{remark}
		Since the weight matrix $M$ in Theorem \ref{thm1} is induced by $L^2(\Omega)$ at the discrete level, $M$ is actually the mass matrix in the context of the finite element method. That is, the $(i,j)$ entry of $M$ is actually $\int_{\Omega }\phi_j\phi_i dxdy, ~1\leq i,j\leq m$, where the $\{\phi_i\}_{i=1}^m$ are the finite element basis functions. 
	\end{remark}
	\begin{remark} 
		One also can interpret $\sum_{j=1}^{n}\tau\|v(t_j)\|_0^2$ as a  discrete $L^2(\tau,T;L^2(\Omega))$ norm, where $v(t)$ is a piecewise constant function over time $(\tau,T]$. Then minimizing $\sum_{j=1}^{n}\tau\|{u}^{j}_{h}-{u}^{j}_{h,r}\|_0^2$ with an optimal low dimension basis is equivalent to finding the POD of the operator $U=\begin{bmatrix}
			\vec{u}_{h}^{1}&\vec{u}_{h}^{2}&...&\vec{u}_{h}^{n}  
		\end{bmatrix}$ with respect to $L^2(\tau,T;L^2(\Omega))$. This approach leads to a different POD formulation but yields the same data compression, see, e.g., \cite{MR3986356}.
	\end{remark}

	%\vspace{1mm}

	\subsubsection{Gradient error analysis for a nonlinear PDE constraint}\label{NGD}
	
	\vspace{2mm}
	The operator $F(\cdot)$ in problem (\ref{Cfunctional}) or (\ref{discretecostChapter3})-(\ref{discreteoptChapter3}) is often nonzero as well. Next, we consider the Navier-Stokes equations (NSE) as a typical example of such a nonlinear constraint to analyze the error introduced by iPOD to the gradient. We briefly recall the NSE:
	\begin{equation}\label{NS}
		\left\{\begin{aligned}
			&\frac{\partial {u}}{\partial t}-\nabla\cdot{\mathbb T}({u},p)+(u\cdot\nabla) u={f} \quad \text{in}~\Omega\times(0,T],\\
			&\nabla\cdot {u}=0\quad \text{in}~\Omega\times(0,T],\\
			&{u}(\cdot,0)={u}_0  \quad \text{in}~~\Omega,\\
			&{u}=0 \quad \text{on}~~\partial\Omega,
		\end{aligned}\right.
	\end{equation}
	where $u=\begin{pmatrix}u_1&u_2\end{pmatrix}^T$ is velocity in tangential and vertical direction, ${\mathbb T}({u},p)=2\nu{\mathbb D}({u})-p{\mathbb I}$ is the stress tensor, ${\mathbb D}({u})=\frac{1}{2}(\nabla {u}+\nabla^T {u})$ is the deformation tensor, $\nu$ is the kinematic viscosity of the fluid, $p$ is the kinematic pressure, and ${f}$ is a general external
	forcing term. 
	At the discrete level, the fully discretized NSE (based on backward Euler and FE spatial discretization) is 
	\begin{equation}
		\label{NSeqnfor}
		\left\{\begin{aligned}
			&\left \langle\frac{ {u}_{h}^{j+1}- {u}_{h}^{j}}{\tau},v_h\right\rangle+ a\left({u}_{h}^{j+1},v_h\right)+b\left({u}_{h}^{j+1},{u}_{h}^{j+1},v_h\right)+(p^{n+1}_h,\nabla\cdot v_h)\\&=\left\langle{f}^{n+1},v_h\right\rangle\quad\forall v_h\in V_h,\\
			&(\nabla\cdot{u}_{h}^{j+1},q_h)=0\quad\forall q_h \in Q_h,\\
			& {u}_{h}^{0}= u_{0,h}^{}\quad \text{in} ~~L^2(\Omega),
		\end{aligned}\right.
	\end{equation}
	for $j=0, 1, \ldots,n-1$. 
	Here, we assume $V_h\times Q_h\subset V\times Q$ is a stable pair of FE spaces that satisfies the inf-sup condition \cite[Chapter 4]{NSEGV}, where $Q=\{q\in L^2(\Omega):\int_{\Omega}qdx=0\}$.  The bilinear form $a(\cdot,\cdot)$ and trilinear form  $b(\cdot,\cdot,\cdot)$
	are defined as: 
	\begin{align*}
		a(u,v)&=\left(\nu{\mathbb D}({u}), {\mathbb D}({v})\right)=\int_{\Omega}\nu{\mathbb D}({u})\cdot {\mathbb D}({v})dxdy\quad \forall u,v\in V_h, \\
		b(u,w,v)&=\left((u\cdot\nabla)w, v\right)=\int_{\Omega}(u\cdot\nabla)w\cdot vdxdy\quad \forall u, w, v\in V_h. 
	\end{align*}
	Using the space $V_h^{div}=\{v\in V_h: (q,\nabla\cdot v)=0~~\forall q\in Q_h\}$ and the inf-sup condition, equation (\ref{NSeqnfor}) is equivalent to 
	\begin{equation}
		\label{NSeqnfordv}
		\left\{\begin{aligned}
			&\left \langle\frac{ {u}_{h}^{j+1}- {u}_{h}^{j}}{\tau},v_h\right\rangle+ a\left({u}_{h}^{j+1},v_h\right)+b\left({u}_{h}^{j+1},{u}_{h}^{j+1},v_h\right)\\&=\left\langle{f}^{n+1},v\right\rangle~\forall v_h\in V_h^{div},\\
			& {u}_{h}^{0}= u_{0,h}^{}\quad \text{in} ~~L^2(\Omega).
		\end{aligned}\right.
	\end{equation}
	Then the operators $A$ and $F(\cdot)$ in the problem (\ref{discretecostChapter3})-(\ref{discreteoptChapter3}) are induced by $a(u,v)=\langle Au,v\rangle~~\forall u,v\in V_h^{div}$ and $b(u,u,v)=\langle F(u),v\rangle~~\forall u,v\in V_h^{div}$.  In addition, the bilinear form $a(\cdot,\cdot)$ has a lower bound by the Korn inequality \cite{Hou2023}: Let $C_K$ be a positive constant depending on $\Omega$, then 
	\begin{align}\label{Korn}
		a(v,v)\geq \nu C_K\|\nabla v\|_{L^2(\Omega)}\|\nabla v\|_{L^2(\Omega)}\quad \forall v\in H^1(\Omega).
	\end{align}
	
	In order to calculate the gradient of the objective function (\ref{discretecostChapter3}) at $u_{0,h}$, we need to first solve the NSE (\ref{NSeqnfor}) forward and then use the forward solution $\{u_h^j\}$ to solve a backward adjoint equation, for $j=n-1, n-2, \cdots,1,0$,
	\begin{equation}
		\label{NSeqnback}
		\left\{\begin{aligned}
			&\left\langle-\frac{ {u}_{h}^{*j+1}-{u}_{h}^{*j}}{\tau},v_h\right\rangle+a\left({u}_{h}^{*j},v_h\right)+b\left({v}_{h},{u}_{h}^{j+1},{u}_{h}^{*j}\right)\\&+b\left({u}_{h}^{j+1},{v}_{h},{u}_{h}^{*j}\right)+(p^{*j}_h,\nabla\cdot v_h)=\left\langle\hat {u}^{j+1}-{u}^{j+1}_h,v_h\right\rangle\quad\forall v_h\in V_h,\\
			&(\nabla\cdot{u}_{h}^{*j},q_h)=0\quad\forall q_h \in Q_h,\\
			&{u}_h^{*n}= 0\quad \text{in} ~~L^2(\Omega).
		\end{aligned}\right.
	\end{equation}
Note that in the above, we have used the fact that the operator $A$ induced by $a(\cdot,\cdot)$ is self-adjoint and $\left\langle \left(F'(u_h^{j+1})\right)^*{u}_{h}^{*j},v_h\right\rangle=b\left({v}_{h},{u}_{h}^{j+1},{u}_{h}^{*j}\right)+b\left({u}_{h}^{j+1},{v}_{h},{u}_{h}^{*j}\right)$. 
	Then the gradient is given by
	\begin{align}\label{NSeqngradient} \nabla J({u}^{}_{0,h})=-{u}^{*0}_{h}+\gamma  {u}^{}_{0,h}.
	\end{align}
	Again, the iPOD must be applied for the forward NSE solution data $\{u_{h}^n\}_{j=1}^n$ to reduce the computer memory requirements. We prove that in this nonlinear problem, the gradient error from iPOD can be bounded if the Hilbert space chosen for the iPOD is $H^1(\Omega)$ instead of $L^2(\Omega)$.   
	\vspace{1mm}   
	
	\begin{theorem}\label{thm2}
		Let $\{{u}_{h}^{j}\}_{j=1}^n$ be solutions of equation (\ref{NSeqnfordv}) with a given initial condition $u_{0,h}\in L^2(\Omega)$ and $\{\vec{u}_{h}^{j}\}_{j=1}^n$	be vector representations of $\{{u}_{h}^{j}\}_{j=1}^n$.
		Consider the POD operator $U=\begin{bmatrix}
			\tau^{\frac{1}{2}}\vec{u}_{h}^{1}&\tau^{\frac{1}{2}}\vec{u}_{h}^{2}&...&\tau^{\frac{1}{2}}\vec{u}_{h}^{n}  
		\end{bmatrix}: \mathbb{R}^{n}\mapsto \mathbb{R}^{m}_M$, where the weight matrix $M$ is induced by $H^1(\Omega)$ norm. Let $\widetilde{U}=\begin{bmatrix}
			\tau^{\frac{1}{2}}\vec{u}_{h,r}^{1}&\tau^{\frac{1}{2}}\vec{u}_{h,r}^{2}&...&\tau^{\frac{1}{2}}\vec{u}_{h,r}^{n}  
		\end{bmatrix}$ be the approximated data matrix of $U$ by applying iPOD and let $\epsilon$ be an error satisfying $\|U-\widetilde{U}\|_\mathrm{HS}\leq \epsilon$. If the time step $\tau$ is small enough,
		then the gradient error $\xi$ introduced via iPOD is bounded by
		\begin{align}
			\|\xi\|_{L^2(\Omega)}\leq  \sqrt{C_4e^{C_3T}} \, \epsilon,
		\end{align}
		where $C_3$ and $ C_4$ are constants depending on $\Omega$ and  $\{{u}_{h}^{*j}\}_{j=1}^{n}$ and $\{{u}_{h,r}^{*j}\}_{j=1}^{n}$, and $\{{u}_{h,r}^{*j}\}$ is the solution of (\ref{NSeqnback}) replacing $\{{u}_{h}^{j+1}\}_{j=0}^{n-1}$ with $\{{u}_{h,r}^{j+1}\}_{j=0}^{n-1}$.
	\end{theorem}
	\begin{remark}
		Theorem \ref{thm1} and \ref{thm2} tell us that even under the same objective function, the Hilbert spaces chosen for iPOD data compression may be different due to different PDE constraints. %A thorough analysis has to be committed to ensure the data error and gradient error are bounded as required.% with respect to the relevant norm.
	\end{remark}
	\begin{remark}
		Since the weight matrix $M$ in Theorem \ref{thm2} is induced by $H^1(\Omega)$, then for instance, in $2D$ NSE, $M=\begin{bmatrix}M_1&0\\0&M_1\end{bmatrix}$, and the $(i,j)$ entry of $M_1$  is $\int_{\Omega }\left(\phi_j\phi_i + \nabla\phi_j\nabla\phi_i\right) dxdy, ~1\leq i,j\leq m$, where $\{\phi_i\}_{i=1}^m$ are the finite element basis functions. 
	\end{remark}
	\begin{proof}
		Consider the backward equation (\ref{NSeqnback}) equipped with the approximated data $\{{u}_{h,r}^{j}\}_{j=1}^n$:
		\begin{equation}
			\label{NSApp}
			\left\{\begin{aligned}
				&\left\langle-\frac{ {u}_{h,r}^{*j+1}-{u}_{h,r}^{*j}}{\tau},v_h\right\rangle+a\left({u}_{h,r}^{*j},v_h\right)+b\left({v}_{h},{u}_{h,r}^{j+1},{u}_{h}^{*j}\right)\\&+b\left({u}_{h,r}^{j+1},{v}_{h},{u}_{h}^{*j}\right)+(p^{*j}_{h,r},\nabla\cdot v_h)=\left\langle\hat {u}^{j+1}-{u}^{j+1}_{h,r},v_h\right\rangle,\\
				&(\nabla\cdot {u}_{h,r}^{*j},q_h)=0,\\
				&{u}_h^{*n}= 0,
			\end{aligned}\right.
		\end{equation}
		where ${u}_{h,r}^{*j}$ is the solution at time $t_j$ solved with the approximated data ${u}_{h,r}^{j+1}$.\\
		Let $e_h^{j}={u}_{h}^{*j}-{u}_{h,r}^{*j}$, we subtract (\ref{NSeqnback}) from (\ref{NSApp}) to get the error equation:
		\begin{equation}
			\label{NSe}
			\left\{\begin{aligned}
				&\left\langle-\frac{ e_h^{j+1}-e_h^{j}}{\tau},v_h\right\rangle+a\left(e_h^{j},v_h\right)+b\left({v}_{h},{u}_{h}^{j+1},{u}_{h}^{*j}\right)+b\left({u}_{h}^{j+1},{v}_{h},{u}_{h}^{*j}\right)\\&-b\left({v}_{h},{u}_{h,r}^{j+1},{u}_{h,r}^{*j}\right)-b\left({u}_{h,r}^{j+1},{v}_{h},{u}_{h,r}^{*j}\right)+(p^{*j}_{h}-p^{*j}_{h,r},\nabla\cdot v_h)\\&=\left\langle{u}^{j+1}_{h,r}-{u}^{j+1}_{h},v_h\right\rangle,\\
				&(\nabla\cdot e_h^{j},q_h)=0,\\
				&e_h^{n}= 0.
			\end{aligned}\right.
		\end{equation}
		By adding and subtracting terms, we rewrite the trilinear terms in (\ref{NSe}) as
		\begin{equation}\label{NSt}
			\begin{split}
				&b\left({v}_{h},{u}_{h}^{j+1},{u}_{h}^{*j}\right)+b\left({u}_{h}^{j+1},{v}_{h},{u}_{h}^{*j}\right)-b\left({v}_{h},{u}_{h,r}^{j+1},{u}_{h,r}^{*j}\right)-b\left({u}_{h,r}^{j+1},{v}_{h},{u}_{h,r}^{*j}\right)\\
				&=b\left({v}_{h},{u}_{h}^{j+1},{u}_{h}^{*j}\right)-b\left({v}_{h},{u}_{h,r}^{j+1},{u}_{h}^{*j}\right)+b\left({v}_{h},{u}_{h,r}^{j+1},{u}_{h}^{*j}\right)\\&~~~~-b\left({v}_{h},{u}_{h,r}^{j+1},{u}_{h,r}^{*j}\right)+b\left({u}_{h}^{j+1},{v}_{h},{u}_{h}^{*j}\right)-b\left({u}_{h}^{j+1},{v}_{h},{u}_{h,r}^{*j}\right)\\&~~~~+b\left({u}_{h}^{j+1},{v}_{h},{u}_{h,r}^{*j}\right)-b\left({u}_{h,r}^{j+1},{v}_{h},{u}_{h,r}^{*j}\right)\\
				&=b\left({v}_{h},{u}_{h}^{j+1}-{u}_{h,r}^{j+1},{u}_{h}^{*j}\right)+b\left(v_h,{u}_{h,r}^{j+1},e_h^{j}\right)+b\left({u}_{h}^{j+1},v_h,e_h^{j}\right)\\&~~~~+b\left({u}_{h}^{j+1}-{u}_{h,r}^{j+1},{v}_{h},{u}_{h,r}^{*j}\right).
			\end{split} 
		\end{equation} 
		\iffalse
		Based on (\ref{NSt}), we write equation (\ref{NSe}) as:
		\begin{equation}
			\label{NSe2}
			\left\{\begin{aligned}
				&\left\langle-\frac{ e_h^{j+1}-e_h^{j}}{\tau},v_h\right\rangle+a^*\left(e_h^{j},v_h\right)+b\left({u}_{h}^{*j},{u}_{h}^{j+1}-{u}_{h,r}^{j+1},v_h\right)+b\left(e_h^{j},{u}_{h,r}^{j+1},v_h\right)\\&+b\left({u}_{h}^{j+1},e_h^{j},v_h\right)+b\left({u}_{h}^{j+1}-{u}_{h,r}^{j+1},{u}_{h,r}^{*j},v_h\right)+(p^{*j}_{h}-p^{*j}_{h,r},\nabla\cdot v_h)=\left\langle{u}^{j+1}_{h,r}-{u}^{j+1}_{h}, v_h\right\rangle,\\
				&(\nabla\cdot e_h^{j},q_h)=0,\\
				&e_h^{n}= 0.
			\end{aligned}\right.
		\end{equation}
		\fi
		Set $v_h=e_h^{j}$ in (\ref{NSe}), since $ e_h^{j}\in V_h^{div}$ and $b\left(\cdot,e_h^{j},e_h^{j}\right)=0$, thus with (\ref{NSt}), we have
		\begin{equation}
			\label{NSe3}
			\left\{\begin{aligned}
				&\left\langle-\frac{ e_h^{j+1}-e_h^{j}}{\tau},e_h^{j}\right\rangle+a\left(e_h^{j},e_h^{j}\right)+b\left(e_h^{j},{u}_{h}^{j+1}-{u}_{h,r}^{j+1},{u}_{h}^{*j}\right)\\&+b\left(e_h^{j},{u}_{h,r}^{j+1},e_h^{j}\right)+b\left({u}_{h}^{j+1}-{u}_{h,r}^{j+1},e_h^{j},{u}_{h,r}^{*j}\right)=\left\langle{u}^{j+1}_{h,r}-{u}^{j+1}_{h},e_h^{j}\right\rangle, \\
				&e_h^{n}= 0.
			\end{aligned}\right.
		\end{equation}
		By using the identity $(a-b)a=\frac{a^2-b^2}{2}+\frac{(a-b)^2}{2}$, Korn's inequality on $a\left(e_h^{j},e_h^{j}\right)$, and Cauchy-Schwartz and Young's inequalities on the right side term, we first have 
		\begin{equation}\label{NSe4}
			\begin{split}
				&\frac{\|e_h^{j}\|_{L^2(\Omega)}^2-\|e_h^{j+1}\|_{L^2(\Omega)}^2}{2}+\frac{\|e_h^{j}-e_h^{j+1}\|_{L^2(\Omega)}^2}{2}+\tau\nu C_K\|\nabla e_h^{j}\|_{L^2(\Omega)}^2\\&+\tau b\left(e_h^{j},{u}_{h}^{j+1}-{u}_{h,r}^{j+1},{u}_{h}^{*j}\right)+\tau b\left(e_h^{j},{u}_{h,r}^{j+1},e_h^{j}\right)+\tau b\left({u}_{h}^{j+1}-{u}_{h,r}^{j+1},e_h^{j},{u}_{h,r}^{*j}\right)\\&\leq\frac{\tau}{2}\|{u}^{j+1}_{h,r}-{u}^{j+1}_{h}\|_{L^2(\Omega)}^2 +\frac{\tau}{2}\|e_h^{j}\|_{L^2(\Omega)}^2.%\leq \|{u}^{j+1}_{h,r}-{u}^{j+1}_{h}\|_0\|e_h^{j}\|_0.
			\end{split} 
		\end{equation} 
		The next task is to properly bound the three trilinear terms in (\ref{NSe4}):
		\begin{subequations}\label{NSe5}
			\begin{align}
				&b\left(e_h^{j},{u}_{h}^{j+1}-{u}_{h,r}^{j+1},{u}_{h}^{*j}\right)\nonumber\\ \qquad &\leq C_1\|\nabla e_h^{j}\|_{L^2(\Omega)}\|\nabla\left({u}_{h}^{j+1}-{u}_{h,r}^{j+1}\right)\|_{L^2(\Omega)}\|\nabla{u}_{h}^{*j}\|_{L^2(\Omega)}\nonumber\\
				&\leq \frac{\nu C_K}{6}\| \nabla e_h^{j}\|_{L^2(\Omega)}^2+\frac{3C_1^2}{2\nu C_K}\|\nabla{u}_{h}^{*j}\|_{L^2(\Omega)}^2\|\nabla\left({u}_{h}^{j+1}-{u}_{h,r}^{j+1}\right)\|_{L^2(\Omega)}^2,\\
				&b\left({u}_{h}^{j+1}-{u}_{h,r}^{j+1},e_h^{j},{u}_{h,r}^{*j}\right)\nonumber\\ \qquad &\leq C_1 \|\nabla\left({u}_{h}^{j+1}-{u}_{h,r}^{j+1}\right)\|_{L^2(\Omega)}\|\nabla e_h^{j}\|_{L^2(\Omega)}\|\nabla {u}_{h,r}^{*j}\|_{L^2(\Omega)}\nonumber\\
				&\leq\frac{\nu C_K}{6}\| \nabla e_h^{j}\|_{L^2(\Omega)}^2+\frac{3C_1^2}{2\nu C_K}\|\nabla{u}_{h,r}^{*j}\|_{L^2(\Omega)}^2\|\nabla\left({u}_{h}^{j+1}-{u}_{h,r}^{j+1}\right)\|_{L^2(\Omega)}^2,\\
				&b\left(e_h^{j},{u}_{h,r}^{j+1},e_h^{j}\right)\nonumber\\ \qquad &\leq C_2\| e_h^{j}\|_{L^2(\Omega)}^{\frac{1}{2}}\|\nabla e_h^{j}\|_{L^2(\Omega)}^{\frac{1}{2}}\|\nabla {u}_{h,r}^{j+1}\|_{L^2(\Omega)}\|\nabla e_h^{j}\|_{L^2(\Omega)}\nonumber\\&= C_2\| e_h^{j}\|_{L^2(\Omega)}^{\frac{1}{2}}\| \nabla e_h^{j}\|_{L^2(\Omega)}^{\frac{3}{2}}\|\nabla {u}_{h,r}^{j+1}\|_{L^2(\Omega)}\nonumber\\
				&\leq  \frac{\nu C_K}{6}\| \nabla e_h^{j}\|_{L^2(\Omega)}^2+\left(\frac{9}{2\nu C_K}\right)^3\frac{C_2^4}{4}\|\nabla {u}_{h,r}^{j+1}\|_{L^2(\Omega)}^4\|e_h^{j}\|_{L^2(\Omega)}^2,
			\end{align} 
		\end{subequations} 
		In the above, we have used the trilinear inequalities \cite[Chapter 6]{WLayton2008} and generalized Young's inequality.
		Combining (\ref{NSe4}) and (\ref{NSe5}) leads to
		\begin{equation}\label{NSe6}
			\begin{split}
				&\frac{\|e_h^{j}\|_{L^2(\Omega)}^2-\|e_h^{j+1}\|_{L^2(\Omega)}^2}{2}+\frac{\|e_h^{j}-e_h^{j+1}\|_{L^2(\Omega)}^2}{2}+\tau \frac{\nu C_K}{2}\|\nabla e_h^{j}\|_{L^2(\Omega)}^2\\&\leq \tau\left(\frac{3C_1^2}{2\nu C_K}\|\nabla{u}_{h}^{*j}\|_{L^2(\Omega)}^2+\frac{3C_1^2}{2\nu C_K}\|\nabla{u}_{h,r}^{*j}\|_{L^2(\Omega)}^2\right)\|\nabla\left({u}_{h}^{j+1}-{u}_{h,r}^{j+1}\right)\|_{L^2(\Omega)}^2\\
				&+\frac{1}{2}\tau\|{u}_{h}^{j+1}-{u}_{h,r}^{j+1}\|_{L^2(\Omega)}^2+\tau\left(\left(\frac{9}{2\nu C_K}\right)^3\frac{C_2^4}{4}\|\nabla {u}_{h,r}^{j+1}\|_{L^2(\Omega)}^4+\frac{1}{2}\right)\|e_h^{j}\|_{L^2(\Omega)}^2\\
				&\leq C_4\tau\|{u}_{h}^{j+1}-{u}_{h,r}^{j+1}\|^2_{H^1(\Omega)}+\tau\left(\left(\frac{9}{2\nu C_K}\right)^3\frac{C_2^4}{4}\|\nabla {u}_{h,r}^{j+1}\|_{L^2(\Omega)}^4+\frac{1}{2}\right)\|e_h^{j}\|_{L^2(\Omega)}^2.
			\end{split} 
		\end{equation} 
		Here, $C_4=\max\left\{{\left(\frac{3C_1^2}{2\nu C_K}\|\nabla{u}_{h}^{*j}\|_{L^2(\Omega)}^2+\frac{3C_1^2}{2\nu C_K}\|\nabla{u}_{h,r}^{*j}\|_{L^2(\Omega)}^2\right),\frac{1}{2}}\right\}$.
		Summing (\ref{NSe6}) over $j=n-1,n-2,...,2,1,0$ gives us
		\begin{equation}\label{NSe7}
			\begin{split}
				&\|e_h^{0}\|_{L^2(\Omega)}^2+\sum_{j=n-1}^{0}\left(\|e_h^{j}-e_h^{j+1}\|_{L^2(\Omega)}^2+\nu C_K\tau\|\nabla e_h^{j}\|_{0}^2\right)\\&\leq 2 \tau\sum_{j=n-1}^{0}\left(\left(\frac{9}{2\nu C_K}\right)^3\frac{C_2^4}{4}\|\nabla {u}_{h,r}^{j+1}\|_{L^2(\Omega)}^4+\frac{1}{2}\right)\|e_h^{j}\|_{L^2(\Omega)}^2\\&+2C_4\tau\sum_{j=n-1}^{0}\|{u}_{h}^{j+1}-{u}_{h,r}^{j+1}\|_{L^2(\Omega)}^2.
			\end{split} 
		\end{equation} 
		Supposing 
		$\tau$ is small enough and using the discrete Gronwall inequality, we have 
		\begin{equation}\label{NSe8}
			\begin{split}
				&\|e_h^{0}\|_{L^2(\Omega)}^2+\sum_{j=n-1}^{0}\left(\|e_h^{j}-e_h^{j+1}\|_{L^2(\Omega)}^2+\nu C_K\tau\|\nabla e_h^{j}\|_{L^2(\Omega)}^2\right)\\&\leq C_4e^{C_3T}\sum_{j=n-1}^{0}\tau\|{u}_{h}^{j+1}-{u}_{h,r}^{j+1}\|_{H^1(\Omega)}^2,
			\end{split} 
		\end{equation} 
		where $
		C_3=\max_{0\leq j\leq n-1}\left\{\frac{2\left(\left(\frac{9}{2\nu C_K}\right)^3\frac{C_2^4}{4}\|\nabla {u}_{h,r}^{j+1}\|_{L^2(\Omega)}^4+\frac{1}{2}\right)}{1-2\tau \left(\left(\frac{9}{2\nu C_K}\right)^3\frac{C_2^4}{4}\|\nabla {u}_{h,r}^{j+1}\|_{L^2(\Omega)}^4+\frac{1}{2}\right)}\right\}$.\\
		Again from (\ref{gradient}), the gradient error $\xi$ is actually $e_h^{0}$ 
		and is bounded by 
		\begin{align}\label{NavierGraidentBound}
			\|\xi\|_{L^2(\Omega)}=\|e_h^0\|_{L^2(\Omega)}\leq \left(C_4e^{C_3T}\sum_{j=n-1}^{0}\tau\|{u}_{h}^{j+1}-{u}_{h,r}^{j+1}\|_{H^1(\Omega)}^2\right)^{\frac{1}{2}}.
		\end{align}
		Based on the POD theory, in order to bound $\left(\sum_{j=n-1}^0\tau\|{u}_{h}^{j+1}-{u}_{h,r}^{j+1}\|_{H^1(\Omega)}^2\right)^{\frac{1}{2}}$, we should apply the iPOD for the scaled solution data $\{{\tau}^{\frac{1}{2}}u_h^j\}_{j=1}^n$ with respect to $H^1(\Omega)$. Consequently, the gradient error has bound
		\begin{align*}
			\|\xi\|_0&\leq\sqrt{C_4e^{C_3T}}\left( \sum_{j=n-1}^{0}\|\tau^{\frac{1}{2}}\vec{u}^{j+1}_{h}-\tau^{\frac{1}{2}}\vec{u}^{j+1}_{h,r}\|_{\mathbb{R}^m_M}^2\right)^{\frac{1}{2}}\\&= \sqrt{C_4e^{C_3T}} \, \|U-\widetilde{U}\|_\mathrm{HS}\leq  \sqrt{C_4e^{C_3T}} \, \epsilon,
		\end{align*}
		which completes the proof.
	\end{proof}
	\begin{remark}
		It is worth mentioning that	the well-posedness of the forward fully discretized NSE (\ref{NSeqnfor}) is standard with $u_{0,h}\in L^2(\Omega)$ and $f\in L^{\infty}(0,T;H^{-1}(\Omega))$. Furthermore, with the FE space $V_h\times Q_h$ (such as P2-P1 Taylor-Hood FE space) and $\tau\leq h$, one can prove that the solution $\{u_h^j\}$ with respect to $H^1(\Omega)$ norm is uniformly bounded by $\|u_{0,h}\|_{L^{2}(\Omega)}$ and $\|f\|_{L^{\infty}(0,T;H^{-1}(\Omega))}$ and independent of $h$ and $\tau$ \cite{Leo2022}. In addition, the wellposedness of both the backward adjoint equations (\ref{NSeqnback}) and (\ref{NSApp}) are not hard to obtain since they are simply linear equations equipped with regular data $\{u_h^j\}$, $\{u_{h,r}^j\}$ 
		and $\{\hat{u}^j\}$. Furthermore, if $\tau\leq h$  the solutions $\{u_h^{*j}\}$ and $\{u_{h,r}^{*j}\}$ can be bounded independent of $h$ and $\tau$ and uniformly bounded by $u_{0,h}\in L^2(\Omega)$, $f\in L^{\infty}(0,T;H^{-1}(\Omega))$, $\{\hat{u}^j\}$, and the iPOD data compression error $\epsilon$. Therefore, the  $C_3$ and $C_4$ appearing in Theorem \ref{thm2} are bounded and well-defined constants.
	\end{remark}
	\begin{remark}
		In inequalities (\ref{In4}) and (\ref{NSe8}), note that $\|e_h^0\|_0$ and the discrete $L^2(0,T;H^1(\Omega))$ error $\tau\sum_{j=n-1}^0\|e_h^n\|^2_{H^1(\Omega)}$ are bounded by the iPOD data compression error.
	\end{remark}

	\subsection{Convergence analysis for the inexact gradient method}\label{gradientcon}
	
	Compared to the usual gradient method, the inexact gradient method contains gradient error in each iteration, which may lead to a lower convergence rate or even divergence. With a given termination tolerance $\delta$, we focus on analyzing the necessary gradient error bound of the inexact gradient method so that it preserves the same convergence rate and also the accuracy of the optimal solution as the usual gradient method; this is the main difference between this work and existing inexact gradient analyses \cite{BD2000,DK2022,BD2016, PolB, Mas94}. This analysis perspective is especially relevant for time-dependent PDE-constrained optimization since each iteration needs to solve two PDEs (forward and backward), which is computationally expensive. When lower convergence rates lead to much more iterations, the computation cost is significantly increased. The following analysis will help avoid that. Note that once the gradient error is derived, we can use the analysis results from Sections \ref{LGD} and \ref{NGD} to accordingly set up the iPOD parameters to achieve such gradient error.

	%There are many work that study the convergence of gradient method with errors \cite{BD2000,DK2022,BD2016, PolB, Mas94}, however, a few pays attention to analyzing the convergence behavior of inexact gradient method.  Compared to the usual gradient method, the convergence rate may decrease and the number of iterations required to converge increases due to the inexact gradient information from each iteration. The number of iterations is especially relevant in PDE-constrained optimization, since each iteration needs  to solve two PDEs (forward and backward) which is super expensive.  In this section, based on the analysis results from Sections \ref{LGD} and \ref{NGD}, we focus on deriving the necessary iPOD  %thresholds  $\texttt{tol}_p$ and $\texttt{tol}_{sv}$ 
	%parameter setup so that the gradient error is well-controlled and the expected number of iterations of the inexact gradient method stays similar to usual gradient method. 

	We start with notation.  Let $ X $ be a Hilbert space and $J:X\mapsto \mathbb{R}$ be a given nonlinear function. Let $J'(x):X\mapsto \mathbb{R}$ be the linear functional that denotes the first order derivative of $J(x)$ at $x \in X$, $\nabla J(x)$ is the gradient or  the dual element of $J'(x)$ in the Hilbert space $X$ so that $ \langle J'(x), y \rangle = ( \nabla J(x), y )_X $ for all $ x, y \in X $.
	
	Recall the inexact gradient method can be generally written as
	\begin{align}\label{ie4}
		x^{(i+1)}=x^{(i)}-\kappa \left(\nabla J(x^{(i)})+\xi^{(i)}\right),
	\end{align}
	where, again, $\kappa$ is the step size of gradient descent and $\xi^{(i)}$ is the gradient error at $i^{th}$ iteration. In our analysis below, we assume the computed gradient is never zero so that $ x^{(i+1)} \neq x^{(i)} $.
	
	%We also need a few smoothness and convexity assumptions on the objective function $J(x)$.
	\begin{definition}
		Let $X$ be a Hilbert space. A function $J : X \mapsto \mathbb{R}$ is convex if
		\begin{align*}
			J(\lambda x + (1-\lambda)y) \leq \lambda J(x) + (1-\lambda)J(y)~~\text{for~~} \forall x, y \in X~~\text{and}~~\lambda\in [0,1].
		\end{align*}
		%for  $\forall x, y \in X$, and $\lambda\in [0,1].$
	\end{definition}

	\begin{lemma}\cite[Appendix B]{BD2000}\label{cov}
		Let $X$ be a Hilbert space.	If the function $J : X \mapsto \mathbb{R}$ is convex and differentiable, then 
		\begin{align}\label{covex}
			J(y)\leq J(x)+\left\langle J'(x),y-x\right\rangle\quad \forall x,y \in X.
		\end{align}
	\end{lemma}

	\begin{definition}
		For a Hilbert space $X$, a function $J : X \mapsto \mathbb{R}$ is $L$-descent if
		\begin{align}\label{Lcondition}
			J(y)\leq J(x)+\left\langle J'(x),y-x\right\rangle+\frac{L}{2}\|y-x\|_X^2\quad \forall x,y \in X.
		\end{align}
		%where $L$ is called the $L$-smooth constant.
	\end{definition}
	
	\begin{remark}
		Some analysis works on inexact gradient methods assume the function J has a L-Lipschitz continuous gradient. However, the L-descent condition is weaker than this gradient assumption; see \cite{DK2022} for more information.
	\end{remark}
	%An $L$-descent function is usually understood as a function that has an $L$-Lipschitz continuity condition on $\nabla J$, which is a quite strong condition in many scenarios. Another sufficient condition to have L-descent is that the function $J(x)-\frac{L}{2}\|x\|_X^2$ is  differentiable and convex \cite{DK2022}, which is a weaker requirement on $J(x)$. 
	\iffalse
	This can be shown with the following argument: define $G(x)=J(x)-\frac{L}{2}\|x\|_X^2$ and apply Lemma \ref{cov} on $G(x)$: 
	\begin{align}\label{e1}
		J(y)-\frac{L}{2}\|y\|^2_X\leq J(x)-\frac{L}{2}\|x\|^2_X+\langle J'(x), y-x\rangle -(Lx,y-x)_X.
	\end{align}
	Applying identity $a(a-b)=\frac{a^2}{2}-\frac{b^2}{2}+\frac{(a-b)^2}{2}$ on the term $-(Lx,y-x)_X$ and reorganizing (\ref{e1}), we have 
	\begin{align}\label{e2}
		J(y)\leq J(x)+\left\langle J'(x),y-x\right\rangle+\frac{L}{2}\|y-x\|_X^2\quad \forall x,y \in X,
	\end{align}
	this is exactly the inequality (\ref{Lcondition}). In the following analysis, we may only state L-descent function instead of $L$-Lipschitz continuity on $\nabla J$ to avoid unnecessary strong differentiable requirement on $J(x)$.
%\end{remark}
\fi
\begin{lemma} \cite[Chapter 2]{GG2024} \label{covv}
Let $X$ be a Hilbert space.	If $J : X \mapsto R$ is differentiable, $L$-descent, and $\inf J(x)>-\infty$, then
\begin{align}\label{gradientbound}
	\frac{1}{2L}\|\nabla J(x)\|_X^2\leq J(x)-\inf J(x)\quad \forall x\in X.
\end{align}
\end{lemma}
Lemma \ref{covv} can be proved by setting $y=x-\frac1L \nabla J(x)$ in (\ref{Lcondition}) and using $\inf J(x)\leq J(x)\quad \forall x\in X$.

\begin{theorem}\label{thm4}
Let $X$ be a Hilbert space. For the inexact gradient method (\ref{ie4}), assume the objective function $J: X\mapsto \mathbb{R}$ % (\ref{discretecostChapter3}) 
is first order differentiable, convex, and $L$-descent with  minimizer $ x^{\star} \in X $. Suppose the constant gradient descent step size $ \kappa $ satisfies $ 0 < \kappa < 1/L $. If for $ i = 0, \ldots, k-1 $ the gradient error in the $ i $th iteration satisfies $\|\xi^{(i)}\|_X<\frac{2-\kappa L}{4-\kappa L}\|\nabla J(x^{(i)})\|_X$, then the objective function is decreasing, i.e., $J(x^{(i+1)})< J(x^{(i)})$. In addition, for a given iteration termination tolerance constant $ \delta $, choose $\epsilon$ so that $ \|\xi^{(i)}\|_X \leq \epsilon <\frac{2-\kappa L}{4-\kappa L}\|\nabla J(x^{(i)})\|_X$ for all $ i$ and $ \delta \geq \|x^{(0)}-x^\star\|_X\epsilon+\frac{\kappa\epsilon^2}{2\eta} > 0 $, where $\eta = 1-\kappa L\in (0,1)$. Then either there exists an $ i $ so that the iteration terminates, i.e., $ J(x^{(i)}) - J(x^\star) < \delta $, or we have
\begin{align}\label{K0}
	J(x^{(k)})-J(x^\star) &\leq \frac{1}{2k\kappa}\|x^{(0)}-x^{\star}\|_X^2 +\|x^{(0)}-x^{\star}\|_X \epsilon+\frac{\kappa\epsilon^2}{2\eta}.
\end{align}
\end{theorem}
\begin{proof}
First, we derive a gradient error bound to ensure the objective function is decreasing. 
Taking $y=x^{(i)}-\kappa(\nabla J(x^{(i)})+\xi^{i})$ and $x=x^{(i)}$ in the L-descent inequality (\ref{Lcondition}), we have  
\begin{equation}\label{fd}
	\begin{split}
		&J(x^{(i)}-\kappa(\nabla J(x^{(i)})+\xi^{i}))
		\leq J(x^{(i)})+\left\langle J'(x^{(i)}),-\kappa(\nabla J(x^{(i)})+\xi^{i})\right\rangle\\&~~~~+\frac{L}{2}\|\kappa(\nabla J(x^{(i)})+\xi^{i})\|_X^2\\
		&=J(x^{(i)})+\left(\nabla J(x^{(i)})+\xi^{i},-\kappa(\nabla J(x^{(i)})+\xi^{i})\right)_X+\frac{L}{2}\|\kappa(\nabla J(x^{(i)})+\xi^{i})\|_X^2\\&~~~~-\left(\xi^{i},-\kappa(\nabla J(x^{(i)})+\xi^{i})\right)_X\\
		&=J(x^{(i)})-(\kappa-\kappa^2\frac{L}{2})\|\nabla J(x^{(i)})+\xi^{i}\|_X^2-\left(\xi^{i},-\kappa(\nabla J(x^{(i)})+\xi^{i})\right)_X\\
		&\leq J(x^{(i)})-\left((\kappa-\kappa^2\frac{L}{2})\|\nabla J(x^{(i)})+\xi^{i}\|_X^2-\kappa\|\xi^{(i)}\|_X\|\nabla J(x^{(i)})+\xi^{i}\|_X\right).
	\end{split} 
\end{equation} 
In order to have $J(x^{(i)}-\kappa(\nabla J(x^{(i)})+\xi^{i}))-J(x^{(i)})< 0$, i.e., $J(x^{(i+1)})< J(x^{(i)})$, we need based on (\ref{fd})
\begin{align}\label{Decrease}
	(\kappa-\kappa^2\frac{L}{2})\|\nabla J(x^{(i)})+\xi^{i}\|_X>\kappa\|\xi^{(i)}\|_X. 
\end{align}
The inequality (\ref{Decrease}) can be achieved if the following condition holds
\begin{align}\label{DCF1}
	(\kappa-\kappa^2\frac{L}{2})\|\nabla J(x^{(i)})+\xi^{i}\|_X&\geq (\kappa-\kappa^2\frac{L}{2})\left(\|\nabla J(x^{(i)})\|_X-\|\xi^{i}\|_X\right)>\kappa\|\xi^{(i)}\|_X,
\end{align} 
which leads to
\begin{align}\label{DCF}
	\|\xi^{(i)}\|_X<\frac{2-\kappa L}{4-\kappa L}\|\nabla J(x^{(i)})\|_X.
\end{align} 

%{\color{red}Next, note that if the iteration terminates then the result is true. Therefore, from now on, assume the iteration does not terminate, i.e., $ J(x^{(i)}) - J(x^\star) \geq \delta $ for all $ i $.}

Next, we show $ \|x^{(i)} - x^{\star}\|^2_X \leq \|x^{(0)} - x^{\star}\|^2_X $ for all $ i $ while $ J(x^{(i)}) - J(x^\star) \geq \delta $. Based on the inexact gradient iterations, we deduce
\begin{align}\label{CT}
    &\|x^{(i+1)}-x^{\star}\|_X^2-\|x^{(i)}-x^{\star}\|_X^2 \nonumber \\
    &=-\|x^{(i+1)}-x^{(i)}\|_X^2-2\kappa\left(\nabla J(x^{(i)})+\xi^{i},x^{(i+1)}-x^{\star}\right)_X \nonumber\\
    &=-\|x^{(i+1)}-x^{(i)}\|_X^2-2\kappa\left(\nabla J(x^{(i)}),x^{(i+1)}- x^{(i)}\right)_X \\
    &~~~~+2\kappa\left(\nabla J(x^{(i)}),x^{\star}-x^{(i)}\right)_X -2\kappa\left(\xi^{(i)},x^{(i+1)} - x^{(i)}\right)_X \nonumber \\
    &~~~~-2\kappa\left(\xi^{(i)},x^{(i)}-x^{\star}\right)_X.
		%&\leq -\|x^{(i+1)}-x^{(i)}\|_X^2+\kappa L\|x^{(i+1)}-x^{(i)}\|_X^2+2\kappa\left (J(x^{(i)})-J(x^{(i+1)})\right)-2\kappa\left(\xi^{(i)},x^{(i+1)}-x^\star\right).
\end{align}
Applying the L-descent inequality (\ref{Lcondition}) on $-2\kappa\left(\nabla J(x^{(i)}),x^{(i+1)}-x^{(i)}\right)$ gives
\begin{equation}\label{L1}
	\begin{split}
		&-2\kappa\left(\nabla J(x^{(i)}),x^{(i+1)}-x^{(i)}\right)_X\\&\leq 2\kappa\left( J(x^{(i)})-J(x^{(i+1)})\right)+\kappa L\|x^{(i+1)}-x^{(i)}\|_X^2.
	\end{split}
\end{equation}
Using the convexity inequality (\ref{covex}) in Lemma \ref{cov} on $2\kappa\left(\nabla J(x^{(i)}),x^{\star}-x^{(i)}\right)$ gives
\begin{align}\label{C1}
	2\kappa\left(\nabla J(x^{(i)}),x^{\star}-x^{(i)}\right)_X\leq 2\kappa\left(J(x^{\star})-J(x^{(i)})\right).
\end{align}
Using Cauchy-Schwartz and Young's inequalities, we obtain upper bounds for the terms $-2\kappa\left(\xi^{(i)},x^{(i+1)}-x^{(i)}\right)_X$ and  $-2\kappa\left(\xi^{(i)},x^{(i)}-x^{\star}\right)_X$ as follows:
\begin{equation}\label{K13}
	\begin{split}
		-2\kappa\left(\xi^{(i)},x^{(i+1)}-x^{(i)}\right)_X&\leq 2\kappa \|\xi^{(i)}\|_X\|x^{(i+1)}-x^{(i)}\|_X\\&\leq \eta\|x^{(i+1)}-x^{(i)}\|_X^2+\frac{\kappa^2\|\xi^{(i)}\|_X^2}{\eta},\\
		-2\kappa\left(\xi^{(i)},x^{(i)}-x^{\star}\right)_X&\leq 2\kappa \|\xi^{(i)}\|_X\|x^{(i)}-x^{\star}\|_X.
	\end{split} 
\end{equation}
Here, recall $ \eta = 1 - \kappa L $ is positive since $\kappa< 1/L$. Combining (\ref{CT}), (\ref{L1}), (\ref{C1}), and (\ref{K13}) leads to
\begin{equation}\label{K15}
	\begin{split}
		&\|x^{(i+1)}-x^{\star}\|_X^2-\|x^{(i)}-x^{\star}\|_X^2\\&\leq  
		-\left(\eta+\kappa L\right)\|x^{(i+1)}-x^{(i)}\|_X^2+\kappa L\|x^{(i+1)}-x^{(i)}\|_X^2+\eta\|x^{(i+1)}-x^{(i)}\|_X^2\\
		&~~~~+\frac{\kappa^2\|\xi^{(i)}\|_X^2}{\eta} +2\kappa\|\xi^{(i)}\|_X\|x^{(i)}-x^\star\|_X+2\kappa\left(J(x^{\star})-J(x^{(i)})\right)\\&~~~~+2\kappa \left(J(x^{(i)})-J(x^{(i+1)})\right) \\
		&\leq 2\kappa\left(\left(J(x^{\star})-J(x^{(i+1)})\right)+\|\xi^{(i)}\|_X\|x^{(i)}-x^\star\|_X+\frac{\kappa\|\xi^{(i)}\|_X^2}{2\eta}\right).
	\end{split} 
\end{equation}

Recall $ \delta \geq \|x^{(0)}-x^\star\|_X\epsilon+\frac{\kappa\epsilon^2}{2\eta}$, and $ \| \xi^{(0)} \|_X \leq \epsilon $. Since the iteration is not terminated, we have $ J(x^{(1)}) - J(x^{\star}) \geq \delta$ and therefore $\left(J(x^{\star})-J(x^{(1)})\right)+\|\xi^{(0)}\|_X\|x^{(0)}-x^\star\|_X+\frac{\kappa\|\xi^{(0)}\|_X^2}{2\eta} \leq 0$ and hence $\|x^{(1)}-x^{\star}\|_X^2\leq \|x^{(0)}-x^{\star}\|_X^2$. By induction, we have that the sequence $\{\|x^{(i)}-x^\star\|_X\}_{i\geq 0}$ is non-increasing for all $i$ satisfying  $J(x^{(i)}) - J(x^{\star})\geq  \delta$.

Third, we prove the bound (\ref{K0}) in Theorem \ref{thm4}. Define the Lyapunov energy function \cite{KB1960, ACW2016}: 
\begin{align}\label{energy}
	E_i=\frac{1}{2\kappa}\|x^{(i)}-x^{\star}\|_X^2+ i\left(J(x^{(i)})-J(x^\star)\right).%+\|\xi\|_X\|x^{(i+1)}-x^{\star}\|_X.
\end{align}
Since the objective function is decreasing under the condition (\ref{DCF}), we hence have with (\ref{K15})
\begin{equation*}\label{K2}
	\begin{split}
		&E_{i+1}-E_{i}=\frac{1}{2\kappa}\|x^{(i+1)}-x^{\star}\|_X^2+ (i+1)\left(J(x^{(i+1)})-J(x^\star)\right)\\&~~~~-\frac{1}{2\kappa}\|x^{(i)}-x^{\star}\|_X- i\left(J(x^{(i)})-J(x^\star)\right)\\
		&\leq i\left(J(x^{(i+1)})-J(x^{i})\right)%+J(x^{(i+1)})-J(x^\star)+J(x^{\star})-J(x^{(i+1)})\\&~~~
		+\|\xi^{(i)}\|_X\|x^{(i)}-x^\star\|_X+\frac{\kappa\|\xi^{(i)}\|_X^2}{2\eta}\\
		&\leq \|\xi^{(i)}\|_X\|x^{(i)}-x^\star\|_X+\frac{\kappa\|\xi^{(i)}\|_X^2}{2\eta}.
	\end{split} 
\end{equation*}
Summing the above inequality from $0$ to $k-1$ and using the definition of the energy in (\ref{energy}), we have 
\begin{equation}\label{K3}
	\begin{split}
		&k\left(J(x^{(k)})-J(x^\star)\right)-\frac{1}{2\kappa}\|x^{(0)}-x^{\star}\|_X^2\leq E_k-E_0\\&=\frac{1}{2\kappa}\|x^{(k)}-x^{\star}\|_X^2+ k\left(J(x^{(k)})-J(x^\star)\right)-\frac{1}{2\kappa}\|x^{(0)}-x^{\star}\|_X^2\\&\leq\sum_{i=0}^{k-1}\left(\|\xi^{(i)}\|_X\|x^{(i)}-x^\star\|_X+\frac{\kappa\|\xi^{(i)}\|_X^2}{2\eta}\right).
	\end{split} 
\end{equation}
Rearranging (\ref{K3}) and using $\|x^{(i)}-x^\star\|_X\leq \|x^{(0)}-x^\star\|_X $, 
we obtain 
\begin{equation}\label{K4}
	\begin{split}
		J(x^{(k)})-J(x^\star)&\leq \frac{1}{2k\kappa}\|x^{(0)}-x^{\star}\|_X^2+\sum_{i=0}^{k-1}\frac{\|\xi^{(i)}\|_X\|x^{(i)}-x^\star\|_X+\frac{\kappa\|\xi^{(i)}\|_X^2}{2\eta}}{k}\\
		&\leq \frac{1}{2k\kappa}\|x^{(0)}-x^{\star}\|_X^2+\|x^{(0)}-x^\star\|_X\epsilon+\frac{\kappa\epsilon^2}{2\eta}.
		%&\leq \frac{1}{2k\kappa}\|x^{(0)}-x^{\star}\|_X^2+\delta.
	\end{split} 
\end{equation}
This completes the proof. 
\end{proof}

\begin{remark}
Since $\|x^{(i)}-x^\star\|_X\leq \|x^{(0)}-x^\star\|_X$ is overestimated especially as $i$ increases, then we typically expect
\begin{align}
	\sum_{i=0}^{k-1}\frac{\|\xi^{(i)}\|_X\|x^{(i)}-x^\star\|_X+\frac{\kappa\|\xi^{(i)}\|_X^2}{2\eta}}{k}\ll  \delta.
\end{align}
\end{remark}

Actually, better convergence behavior can be proved for the inexact gradient method once stronger smoothness or convexity assumptions are provided on the objective function.
\begin{definition}
Let $X$ be a Hilbert space, and $\mu$ be a positive constant. A function $J : X \mapsto \mathbb{R}$ is said to be $\mu$-{\bf Polyak-Lojasiewicz} if it is bounded below and 
\begin{align}\label{PL}
	J(y)-\inf J(x)\leq \frac{1}{2\mu}\|\nabla J(x)\|_X^2\quad \forall x,y \in X. 
\end{align}
\end{definition}

\begin{definition}
Let $X$ be a Hilbert space, and $\mu$ be a positive constant. A function $J : X \mapsto \mathbb{R}$ is said to be $ \mu $-strongly convex if
\begin{align}\label{SC}
	J(y)\geq J(x)+\left\langle J'(x),y-x\right\rangle+\frac{\mu}{2}\|y-x\|_X^2\quad \forall x,y \in X. 
\end{align}
\end{definition}

\begin{theorem}\label{pl}
Let $X$ be a Hilbert space. For the inexact gradient method (\ref{ie4}), assume the objective function $J: X\mapsto \mathbb{R}$ % (\ref{discretecostChapter3}) 
is first order differentiable, $\mu$-{\bf Polyak-Lojasiewicz}, and $L$-descent with  minimizer $ x^{\star} \in X $. Denote $\theta=1-\mu(2\kappa-L\kappa^2)$.
%Let $X$ be a Hilbert space. Assume that objective function  $J: X \mapsto \mathbb{R}$ is  first order differentiable, $\mu$-{\bf Polyak-Lojasiewicz}, and $L$-descent. Let $x^{\star}$ be the minima of $J(x)$ and 
For $\kappa=1/L$, set $\epsilon$ so that the gradient error $\|\xi^{(i)}\|_X\leq \epsilon$ for all $i=0,1,2,\cdots,k-1$, then the inexact gradient method (\ref{ie4}) satisfies
\begin{align}\label{pl1}
	J(x^{(k)})-J(x^{\star})\leq \left(1-\frac{\mu}{L}\right)^{k}\left(J(x^{(0)})-J(x^{\star})\right)+\frac{1}{2\mu}\left(1-\left(1-\frac{\mu}{L}\right)^k\right)\epsilon^2.
\end{align}
For $0 < \kappa < 1/L$ and $0<\theta<1$, %let $\|\xi^{(i)}\|_X\leq\frac{2-\kappa L}{4-\kappa L}\|\nabla J(x^{(i)})\|_X$ and 
set $\epsilon$ so that $\|\xi^{(i)}\|_X\leq\epsilon<\frac{2-\kappa L}{4-\kappa L}\|\nabla J(x^{(i)})\|_X$ for all $i=0,1,2,\cdots,k-1$, then the inexact gradient method (\ref{ie4}) satisfies
\begin{align}\label{pl2}
	\begin{split}
		J(x^{(k)})-J(x^\star)&\leq\theta^{k}\left(J(x^{(0)})-J(x^{\star})\right)\\&+\frac{1-\theta^k}{1-\theta}\left(\sqrt{2L}|L\kappa^2-\kappa|\sqrt{J(x^{(0)})-J(x^{\star})}\epsilon+\frac{L\kappa^2}{2}\epsilon^2\right).
	\end{split}
\end{align}
\end{theorem}
\begin{proof}
See Appendix \ref{AP1}.
\end{proof}

\vspace{2mm}
\begin{theorem}\label{sc}
Let $X$ be a Hilbert space. For the inexact gradient method (\ref{ie4}), assume the objective function $J: X\mapsto \mathbb{R}$ % (\ref{discretecostChapter3}) 
is first order differentiable, $\mu$-strongly convex, and $L$-descent with  minimizer $ x^{\star} \in X $. Denote $\theta=1-\mu\kappa$ and $\eta=2\kappa-2\kappa^2L$, and suppose the constant gradient descent step size $ \kappa $ satisfies $ 0 < \kappa < 1/L$. For a given iteration termination tolerance constant $\delta $, set $\epsilon$ so that  $\|\xi^{(i)}\|_X \leq \epsilon $ for all $ i =0,1,2,\cdots,k-1$ and $ \delta \geq \frac{\sqrt{2L\eta}+\sqrt{2L\eta+2L\eta\kappa\mu +\kappa\mu}}{\mu\sqrt{2L\eta}}\epsilon > 0$, then the inexact gradient method (\ref{ie4}) satisfies

\begin{align}\label{sc1}
	\begin{split}
		% \|x^{(k)}-x^\star\|_X\leq \theta^{k} \|x^{(0)}-x^\star\|_X+ \frac{1-\theta^{k}}{1-\theta}\|x^{(0)}-x^\star\|_X\kappa\eps.
		\|x^{(k)}-x^\star\|_x^2&\leq \theta^k\|x^{(0)}-x^\star\|_x^2+\frac{1-\theta^k}{1-\theta}\left(2\kappa\|x^{(0)}-x^\star\|_X\epsilon+\frac{\kappa^2 }{2L\eta}\epsilon^2+\kappa^2\epsilon^2\right).
	\end{split}
\end{align}
\end{theorem}
\begin{proof}
See Appendix \ref{AP2}.
\end{proof}
\iffalse
\vspace{2mm}
\begin{theorem}\label{scd}
Let $X$ be a Hilbert space.   Assume that the objective function $J(x): X \mapsto R$ is $\mu$-strongly convex and twice differentiable with $\|\nabla^2 J(x)\|\leq L$. 
If the gradient error  $\|\xi^{(i)}\|_X\leq \epsilon$ for all $i$ and $\theta=\max\{|1-\mu\kappa|, |1-L\kappa|\}<1$, then inexact gradient method (\ref{ie4}) satisfies
\begin{align}\label{scd1}
	\|x^{(k)}-x^\star\|_X\leq \theta^{k} \|x^{(0)}-x^\star\|_X+ \frac{1-\theta^{k}}{1-\theta}\kappa\epsilon.
\end{align}
\end{theorem}
\begin{proof}
See Appendix \ref{AP3}.
\end{proof}
\fi	

The first term on the right hand side of inequalities (\ref{K0}), (\ref{pl1}), (\ref{pl2}), and (\ref{sc1}) gives the classical convergence rate of the usual gradient method, while the second term involving $\epsilon$ are from the inexact gradient descent iterations.   According to a given iteration termination tolerance $\delta$, Theorems \ref{thm4}, \ref{pl}, and \ref{sc} conclude that, if the iPOD parameters are set up properly so that the gradient error $\epsilon$ is controlled and the second right-hand side term in the aforementioned inequalities is approximately the same magnitude as $\delta$, the inexact gradient method is expected to achieve the same level accuracy of the optimal solution while using similar number of iterations compared to the usual gradient method. Also, note that the iPOD data compression algorithm only involves small scale matrix and vector multiplications, which is computationally cheap. Hence, the inexact gradient method overall will not largely increase the computation time.

\section{Numerical tests}\label{NE}

\subsection{Linear PDE-constraint optimization test: Data assimilation for a parabolic interface equation}\label{pbda}
The parabolic interface equations are given by %model many physics and engineering problems when two or more distinct materials or fluids with different conductivities or diffusions are involved, which is given by 
\begin{align}\label{pb1}
&u_{t}-\nabla \cdot (\beta(x,y)\nabla u)=f~ \text{in} ~\Omega\times(0,T];~u=0~ \text{on}~\partial\Omega\times(0,T];~ u(\cdot,0)=u_{0}~\text{in}~\Omega,
\end{align}
together with jump interface condition
$
[u]|_{\Gamma}=g_1,~
[\beta(x,y)\frac{\partial{u}}{\partial\vec{n}}]|_{\Gamma}=g_2.
$
Here, $\Gamma$ is a smooth curve interface that separates an open bounded domain $\Omega$ into two subdomains $\Omega^{+}$ and $\Omega^{-}$ such that $\Omega=\Omega^{+}\cup\Omega^{-}\cup\Gamma$, $[u]|_{\Gamma}=u^{+}|_{\Gamma}-u^{-}|_{\Gamma}$ is the jump of the function $u$ across the interface $\Gamma$, $u^{+}=u|_{\Omega^{+}}$ and $u^{-}=u|_{\Omega^{-}}$,
$\vec{n}$ is the unit normal vector along the interface $\Gamma$ pointing to $\Omega^{-}$, $\frac{\partial{u}}{\partial\vec{n}}$ is the normal derivative of $u$, and $\beta(x,y)$ is assumed to be a positive piecewise function
\begin{align*}
\beta(x,y)& =
\begin{cases}
	\beta^{+}(x,y) & \text{if } (x,y)\in \Omega^{+},\\
	\beta^{-}(x,y)& \text{if } (x,y)\in \Omega^{-}.
\end{cases}
\end{align*}
In the following experiment, we set the model parameters as 
$\Omega^{+}=(0,1)\times(0,1)$, $\Omega^{-}=(1,2)\times(0,1)$, $T=1$,
$\Gamma: x=1$, $g_1=g_2=0$, $\beta^{+}=1$, $\beta^{-}=\frac{1}{2}$, and 
\begin{align*}
f=\begin{cases}
	f^{+}=ty+x^{\frac{1}{2}}+5\quad \text{in~~}\Omega^+\times (0,T],\\
	f^{-}=tx+(xy)^{\frac{1}{2}}+6\quad \text{in~~}\Omega^-\times (0,T].
\end{cases}
\end{align*}
In discrete level, we use the backward Euler scheme with $\tau=\frac{1}{500}$ for the temporal discretization and the P2 confirming finite element space $V_h$ with uniform mesh $h=\frac{1}{50}$ for the spatial discretization.

The data assimilation problem is stated as
\begin{eqnarray}\label{dp1}
\min_{u_{0,h}\in V_h}J_{}(u_{0,h})=\frac{1}{2}\tau \sum_{j=1}^{n}\|\hat u^{j}-u_{h}^{j}\|^{2}_{L^2(\Omega)}+\frac{\gamma}{2}\|u_{0,h}\|^{2}_{L^2(\Omega)},
\end{eqnarray}
subject to 
\begin{align}\label{dc1}
\langle \frac{u^{j+1}_h-u^{j}_h}{\tau},v_h\rangle+a( u^{j+1}_h,v_h)=\langle f^{j+1},v_h\rangle\quad \forall v_h\in V_h;\quad u^0_h=u_{0,h}\quad \text{in}~L^2(\Omega),
\end{align}
for all $j=0,1,\cdots,n$ and $n=\frac{T}{\tau}$. Here, the bilinear form $a(w,v)$ is defined by $a(w,v)=\int_{\Omega^+}\beta^+w^{+}vdxdy+\int_{\Omega^-}\beta^-w^{-}vdxdy\quad \forall w,v\in V_h$. According to problem (\ref{discretecostChapter3})-(\ref{discreteoptChapter3}), the operator $F(\cdot)$ here is $0$ and the operator $A$ is induced by $\langle Aw,v\rangle=a(w,v)\quad \forall w,v\in V_h$. It is not difficult to verify that $A=A^*$ and  $\langle A^*v,v\rangle +C_1\|v\|^2_{L^2(\Omega)}\geq C_2\|v\|_{H^1(\Omega)}^2$.

We choose the observations $\{\hat u^{j}\}$ in (\ref{dp1}) as a noisy numerical solution. Specifically, we construct $\{\hat u^{j}\}$ by adding normal noise $e^j_{ob} \sim \mathcal{N}\left(0,\,\left(\frac{1}{20}I\right)^2\right),$ $j=1,\cdots,N$ onto a numerical solution $\{u_h^j\}$, where the numerical solution is generated by solving the regular parabolic interface equation (\ref{dc1}) with the specific initial data 
$u^{+}_{0,h}=\sqrt{xy(2-x)(1-y)} $ in $\Omega^+$ and $
u^{-}_{0,h}=\sqrt{xy(2-x)(1-y)} $ in $\Omega^-$.

We use the gradient Algorithm \ref{SD} and inexact gradient Algorithm \ref{iPODSD}, respectively, to solve for the optimal solution $u_{0,h}^{\star}$ and the corresponding $\{u^{j\star}_h\}$ for problem (\ref{dp1})-(\ref{dc1}), and test the effectiveness of the inexact gradient method for saving data storage. We set the parameter $\gamma=\frac{1}{2000}$, the step size $\kappa=1$, and the steepest descent termination criteria $\|\nabla _{}J(u_{0,h}^{(k)})\|_{L^(\Omega)}\leq \texttt{tol}_{sd}=10^{-5}$. Also, since $A$ is linear and differentiable and the $L^2(\Omega)$ norm is convex and differentiable, the problem (\ref{dp1})-(\ref{dc1}) is a linear-quadratic optimization that has enough smoothness and strong convexity. Based on Theorems \ref{thm0}, \ref{thm1} and \ref{sc}, to avoid affecting the solution accuracy and the convergence speed from inexact gradient method, we should apply the iPOD with respect to $M$ (that is the mass matrix induced by $L^2$ norm) and the SVD truncation threshold $\texttt{tol}_p$ and $\texttt{tol}_{sv}$ as $10^{-8}$, since we want to have $n \texttt{tol}_p<10^{-5}$ and $n \texttt{tol}_{sv}<10^{-5}$ so that the data compression error and gradient error are less than $\texttt{tol}_{sd}=10^{-5}$.

\begin{figure}%[H]
\centering
\centerline{\includegraphics[height=4.5cm,width=4.2cm]{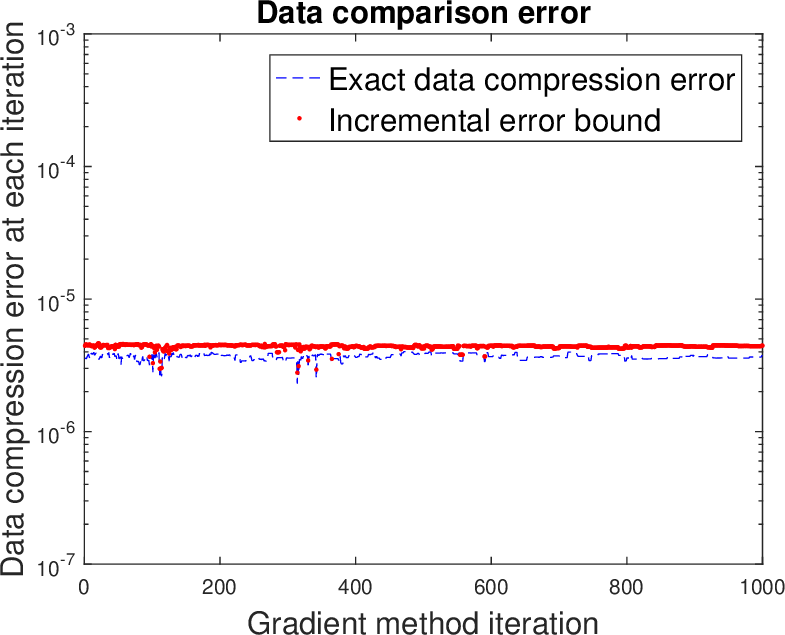}\includegraphics[height=4.5cm,width=4.2cm]{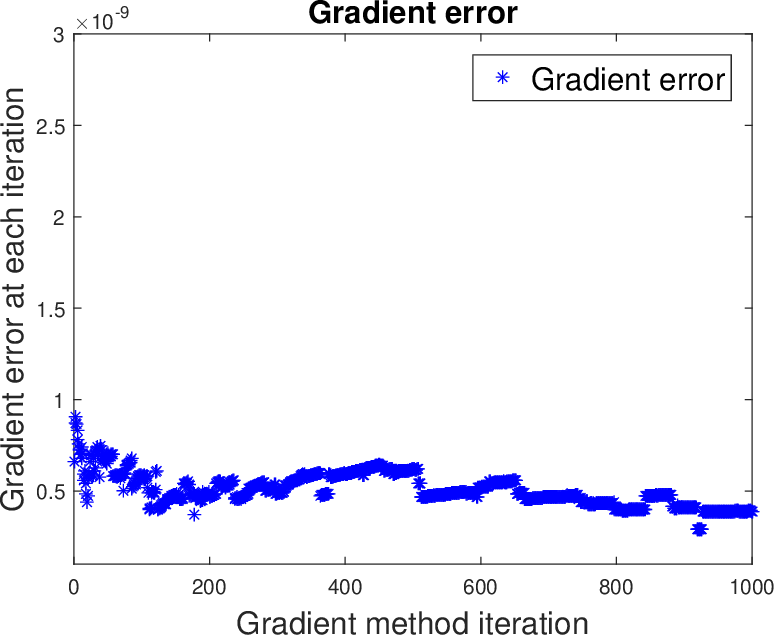}\includegraphics[height=4.5cm,width=4.2cm]{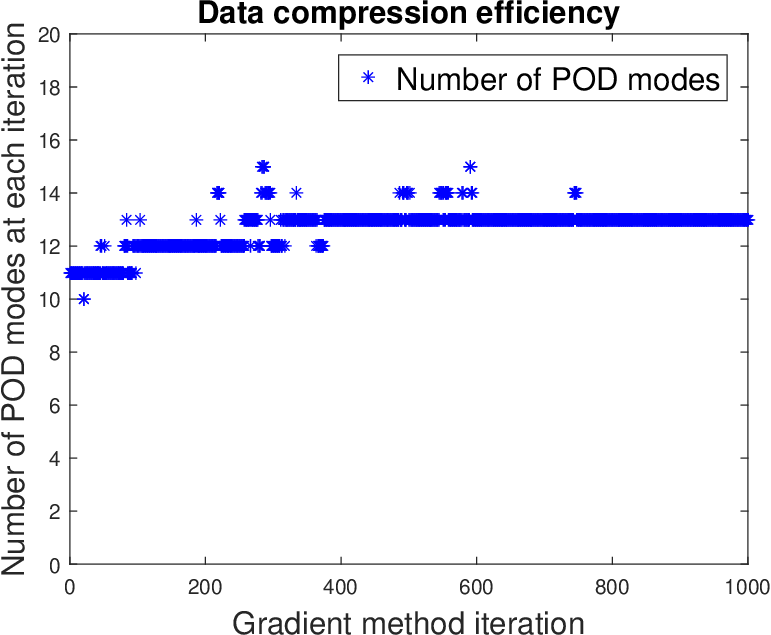}}
\caption{Linear PDE test 1: Data compression performance}\label{NC2}
\end{figure}

\begin{table}%[H]
\centering
\caption{Linear PDE test 1: Comparison between the usual gradient method and the inexact gradient method.  $\{u^j_h\}$: exact solution; $\{u^{j\star}_h\}$: optimal solution; Relative error: $\|u_h-u_h^{\star}\|_{\tilde{L}^2}=\sqrt{\sum_{j=1}^n\tau\|u_h^j-u_h^{j\star}\|_{L^2(\Omega)}^2/\|u_h^j\|^2_{L^2(\Omega)}}$.}
\begin{tabular}{c|c|c}
	\hline
	%\multicolumn{3}{c}{Error Comparison to Observations}\\ \cline{1-3}
	&  usual gradient method & inexact gradient method\\
	\hline
	number of iteration	& 1001&1001 \\ \hline
	data storage	& $10201\times500$& $\approx 10201\times 13$  \\ \hline
	%	memory saved	 &0&$96.0\%$\\ \hline
	$\|u_h-u_h^{\star}\|_{\tilde{L}^2}$	& $5.2306\times 10^{-3}$&$5.2306\times 10^{-3}$ \\\hline
\end{tabular}
\end{table}\label{parabolic}

In  figure \ref{NC2}, the first plot shows that the error (red dash line) calculated with Theorem \ref{thm0} is less than $10^{-5}$, and the exact data compression error (blue dash line) is sharply bounded by the red line, which verifies that Theorem \ref{thm0} provides an accurate error estimate for iPOD Algorithm; the second plot shows that the gradient error at each iteration is strictly less than $10^{-5}$ %(the convergence tolerance $\texttt{tol}_{sd}$ and the gradient error $\|\xi^i\|_{L^2(\Omega)}$) 
and thus well-controlled as expected, this verifies Theorem \ref{thm1}; the third plot shows that the inexact gradient method only uses $\frac{13}{500}\times 100 $ ($\approx 2.6$) percent storage of the usual gradient method. Table \ref{parabolic} displays that the inexact gradient method iterates the same number of steps and achieves the same accuracy of the optimal solution compared to the usual gradient method. All these prove the memory efficiency of the inexact gradient method, and also confirms the conclusions in Section \ref{SDA} (Theorems \ref{thm1} and \ref{sc}) that, once the iPOD data compression error and gradient error are properly controlled, the inexact gradient method will have minimal impact on the convergence speed and the solution accuracy.
\begin{figure}%[H]
\centering
{\includegraphics[height=4.5cm,width=4.2cm]{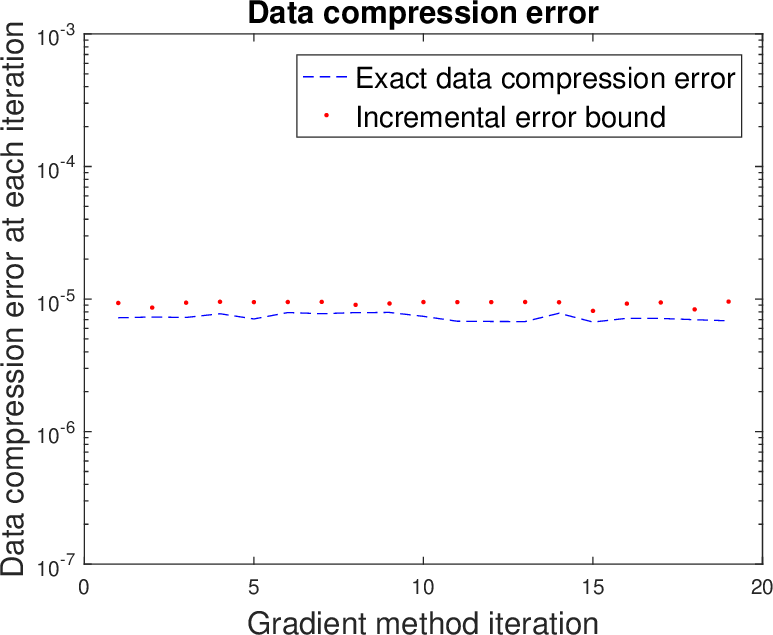}\includegraphics[height=4.5cm,width=4.2cm]{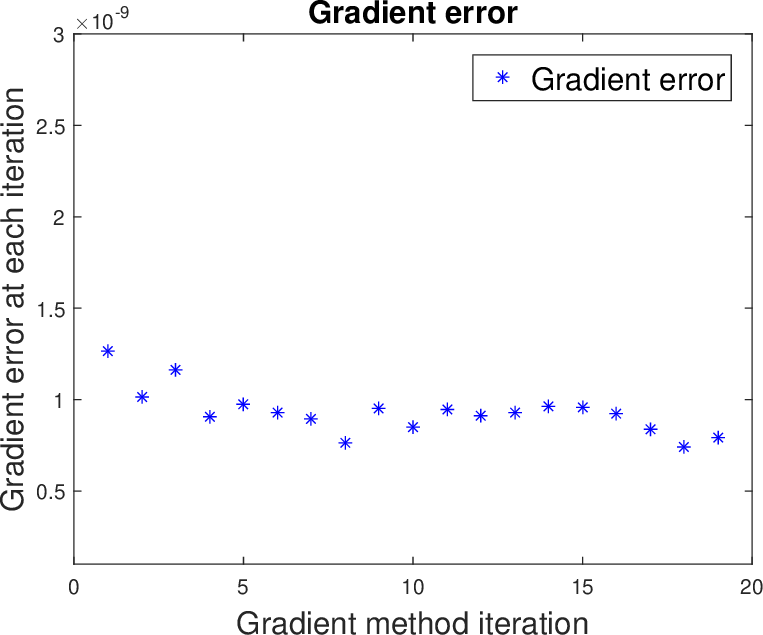}\includegraphics[height=4.5cm,width=4.2cm]{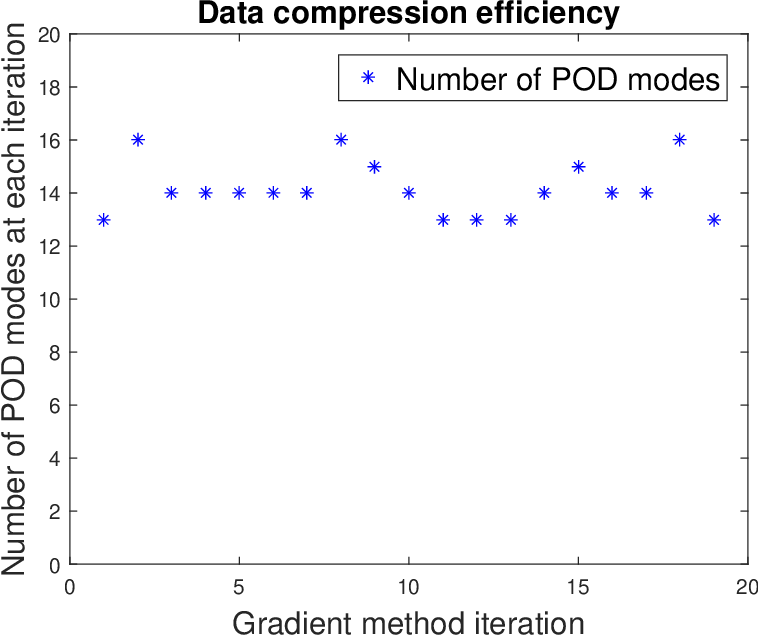}}
\caption{Linear PDE test 2: Data compression performance }\label{NC3}
\end{figure}

Moreover, to further test the memory efficiency of the inexact gradient method and the analysis results from Section \ref{SDA}, we run the optimization problem (\ref{dp1})-(\ref{dc1}) in a larger scale for a few gradient descent iterations and observe the behaviors. We set $h=\frac{1}{50}$ and $\tau=\frac{1}{2000}$ (with data size: $20402\times 2000$) and $\|\nabla _{}J(u_{0,h}^{(k)})\|_{L^(\Omega)}\leq \texttt{tol}_{sd}=10^{-5}$.   Accordingly, we set $\texttt{tol}_p=\texttt{tol}_{sv}=5\times10^{-9}$ to control the data compression error and gradient error. The other parameters are set the same as above. The first two plots in figure \ref{NC3} show the iPOD data compression error and gradient error are fully controlled as expected by $\texttt{tol}_p$ and $\texttt{tol}_{sv}$. The third plot in Figure \ref{NC3} shows that only $\frac{15}{2000}\times 100$ ($\approx 0.75$) percent data storage is used compared to usual gradient method. Since the gradient error at each iteration is strictly less than $10^{-5}$, we can expect that with Theorems \ref{sc}, the inexact gradient method will not  impact the convergence speed and solution accuracy. 

\subsection{Nonlinear PDE-constraint optimization test: Data assimilation for NSE}\label{nda}
We consider a 2D NSE in time-space domain $[0,1]\times\{[0,1]\times[-1,0]\}$ with samll viscosity $\nu=\frac{1}{1000}$. We use the backward Euler scheme with $\tau=\frac{1}{500}$ to discretize time and the P2 FE space $V_h$ and the P1 FE space $Q_h$ with uniform mesh $h=\frac{1}{20}$ to discretize the velocity and pressure.

The data assimilation problem is stated as: 
\begin{align}\label{NS1}
\min_{u_{0,h}\in V_h}J_{}(u_{0,h})=\frac{1}{2}\tau \sum_{j=1}^{n}\|\hat u^{j}-u_{h}^{j}\|^{2}_{L^2(\Omega)}+\frac{\gamma}{2}\|u_{0,h}\|^{2}_{L^2(\Omega)}\quad\text{subject to}\quad (\ref{NSeqnfor}).
\end{align}
\iffalse
subject to 
\begin{equation}\label{dc2}
\left\{\begin{aligned}
	&\left \langle\frac{ {u}_{h}^{j+1}- {u}_{h}^{j}}{\tau},v_h\right\rangle+ a\left({u}_{h}^{j+1},v_h\right)+b\left({u}_{h}^{j+1},{u}_{h}^{j+1},v_h\right)+(p^{n+1}_h,\nabla\cdot v_h)\\&=\left\langle{f}^{n+1},v_h\right\rangle\quad\forall v_h\in V_h,\\
	&(\nabla\cdot{u}_{h}^{j+1},q_h)=0\quad\forall q_h \in Q_h,\\
	& {u}_{h}^{0}= u_{0,h}^{}\quad \text{in} ~~L^2(\Omega).
\end{aligned}\right.
\end{equation}
for all $j=0,1,\cdots,n$ and $n=\frac{T}{\tau}$.
\fi
The observations $\{\hat u^{j}\}$ in (\ref{NS1}) are given by adding normal noise $e^n_{ob} \sim \mathcal{N}\left(0,\,\left(\frac{1}{20}I\right)^2\right),$ $j=1,2,3,\cdot\cdot\cdot, N$ onto a numerical solution $\{ u^{j}_h\}$, where the numerical solution is generated by running the regular NSE (\ref{NSeqnfor}) with initial condition $u_{0}=\begin{pmatrix}u_{1,0}&u_{2,0}\end{pmatrix}^T$. Here, $u_{1,0}=5xy(1-x)(1+y)$ and $u_{2,0}=4xy(1-x)(1+y)$.

%For the gradient methods, we set $\gamma=\frac{1}{1000}$, $\kappa=1$, and $\texttt{tol}_{sd}= 10^{-4}$. 
For the inexact gradient method, based on Theorem \ref{thm2}, we should apply the iPOD with respect to the matrix $M$ that is induced by the $H^1$ norm to control the gradient error. Although the objective function (\ref{NS1}) is smooth, there is not a guarantee that (\ref{NS1}) is strongly convex due to the nonlinear constraint. We thus can only assume it is convex and L-descent, which corresponds the case of Theorem \ref{thm4}. If we set the steepest descent termination criteria as $\texttt{tol}_{sd}=|J(u_{0,h}^{k+1}-J(u_{0,h}^{k})|\leq10^{-8} $, to avoid hurting the convergence speed and solution accuracy, we need to apply truncation thresholds $\texttt{tol}_p=\texttt{tol}_{sv}=2\times10^{-11}$ in iPOD so that the data compression error ($n\texttt{tol}_p$ and $n\texttt{tol}_{sv}$) and gradient error do not exceed magnitude $10^{-8}$. We also set the parameters $\gamma=\frac{1}{1000}$ and $\kappa=1$.
%Different from the test in Section \ref{pbda}, based on Theorem \ref{thm2}, we should apply the iPOD Algorithm with respect to matrix $M$ that is induced by $H^1$ norm instead of $L^2$ norm to control the gradient error and avoid ill-effect on the accuracy and convergence speed from inexact gradient method.

\begin{figure}%[H]
\centering
{\includegraphics[height=4.5cm,width=4.2cm]{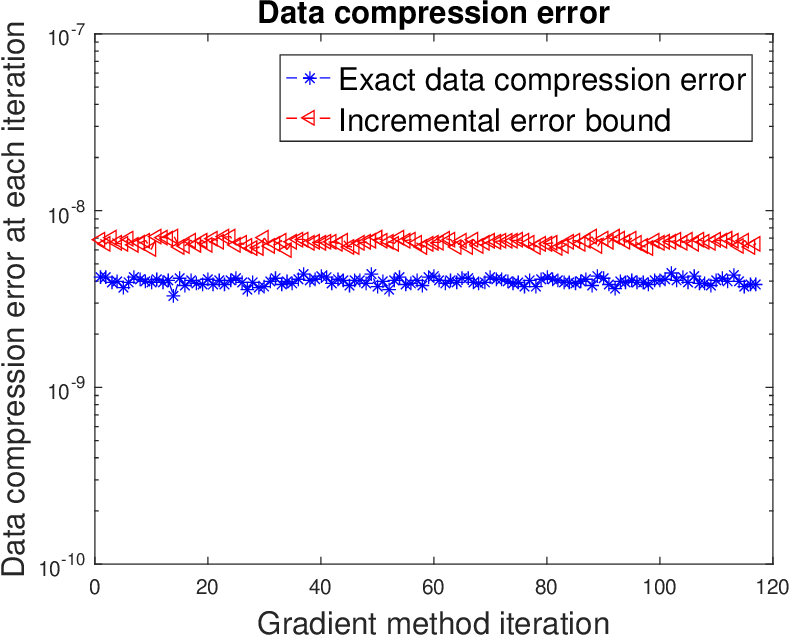}\includegraphics[height=4.5cm,width=4.2cm]{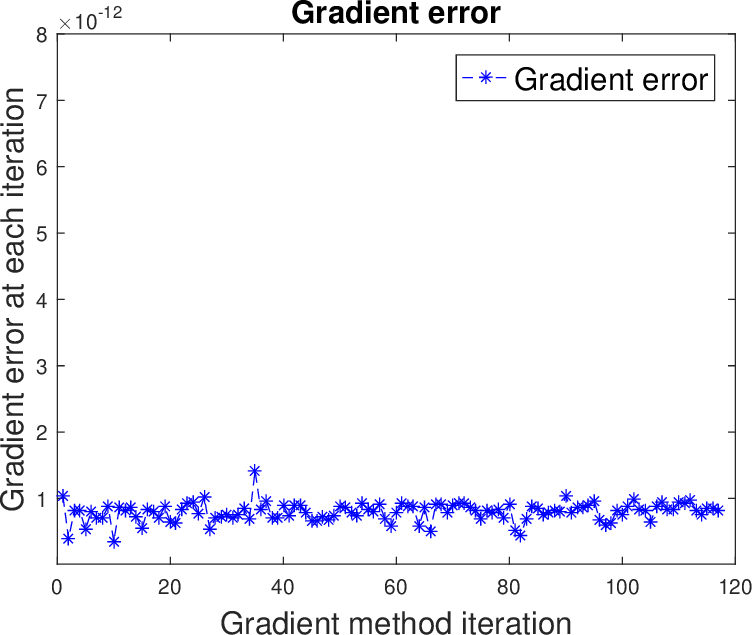}\includegraphics[height=4.5cm,width=4.2cm]{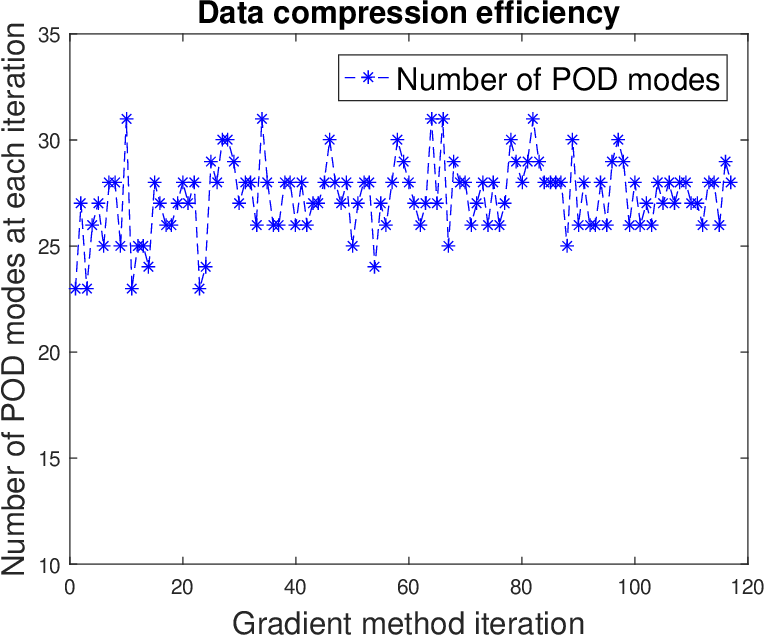}}
\caption{Nonlinear NSE test 1: Data compression performance }\label{NST2}
\end{figure}

\begin{table}%[H]
\centering
\caption{Nonlinear NSE test 1: Comparison between the usual gradient method and the inexact gradient method.  $\{u^j_h\}$: exact solution; $\{u^{j\star}_h\}$: optimal solution; Relative error: $\|u_h-u_h^{\star}\|_{\tilde{L}^2}=\sqrt{\sum_{j=1}^n\tau\|u_h^j-u_h^{j\star}\|^2_{L^2(\Omega)}/\|u_h^j\|^2_{L^2(\Omega)}}$. %The memory saved is calculated by $1-\frac{data~column~needed}{501}\cdot 100\%$.
}\label{tablePODST22}
\begin{tabular}{c|c|c}
	\hline
	%\multicolumn{3}{c}{Error Comparison to Observations}\\ \cline{1-3}
	&  usual gradient method & inexact gradient method\\
	\hline
	number of iteration	& 117&117 \\ \hline
	data storage	& $3362\times 500$& $\approx 3362\times 27$  \\ \hline
	%	memory saved	 &0&$96.0\%$\\ \hline
	$\|u_h-u_h^{\star}\|_{\tilde{L}^2}$	& $7.14467
	\times 10^{-3}$&$7.14467\times 10^{-3}$ \\\hline
\end{tabular}
\end{table}\label{NSET}

In figure \ref{NST2}, the first plot again shows the exact data compression error is sharply bounded by the iPOD error estimate that is calculated with Theorem \ref{thm0}; the second plot shows that the gradient error is strictly less than $10^{-8}$ at each iteration as expected, which can always be well bounded by the controllable iPOD data compression error; the third plot shows that the inexact gradient method uses around $\frac{27}{500}\times 100$ ($\approx 5.4$) percent storage of the usual gradient method. In addition, Table \ref{NSET} displays that the inexact gradient method achieves the same solution accuracy and convergence speed as the usual gradient method but with much less data storage.

Next, by setting $h=\frac{1}{25}, \tau=\frac{1}{1000}$,  $\texttt{tol}_p=\texttt{tol}_{sv}=10^{-11}$, $\texttt{tol}_{sd}=10^{-8}$, and  other parameters the same as above, we run the inexact gradient method for a few iterations in a larger scale with data size $5202\times 1000$ to further test its memory efficiency. Figure \ref{NST3} shows that the data compression and gradient errors are bounded as desired so that it has minimal impact on solution accuracy and convergence speed. Also, the storage used in the inexact gradient method is around $\frac{32}{1000}\times 100$ ($\approx 3.2$) percent of the usual gradient method. All these tests confirm that the inexact gradient method with iPOD is memory efficient and robust. 

\begin{figure}%[H]
\centering
{\includegraphics[height=4.5cm,width=4cm]{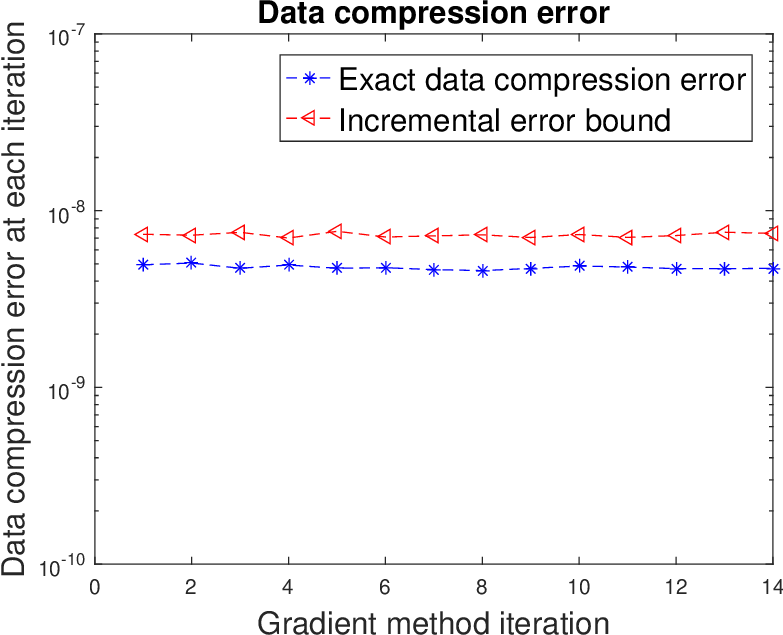}\includegraphics[height=4.5cm,width=4cm]{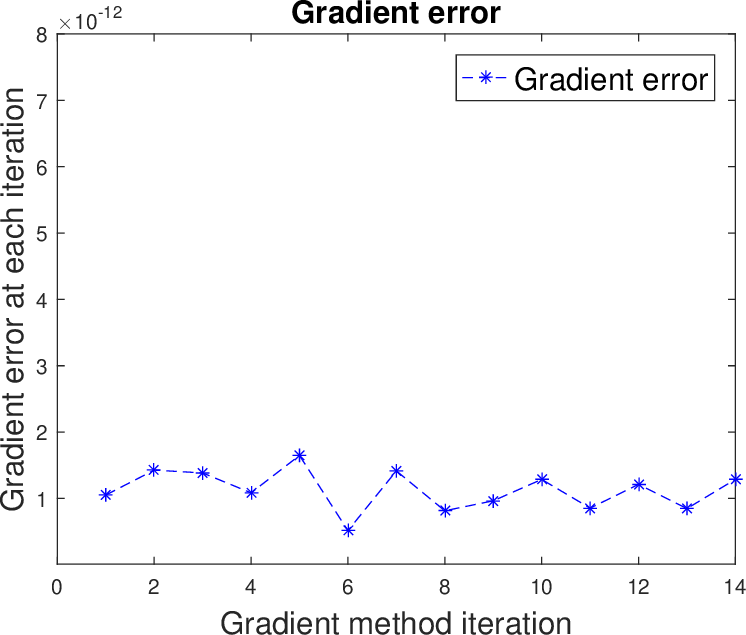}\includegraphics[height=4.5cm,width=4cm]{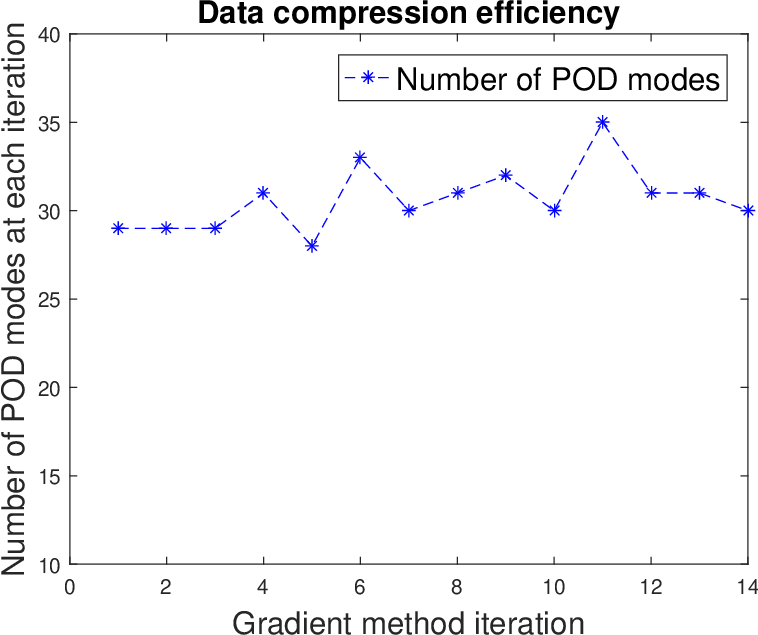}}
\caption{Nonlinear NSE test 2: Data compression performance }\label{NST3}
\end{figure}

\begin{remark}
It is worth noticing that the gradient errors in tests from Section \ref{pbda} and \ref{nda} are usually smaller than required, which means in practice we can apply the iPOD truncation thresholds more aggressively to further reduce the data storage required. One simple strategy is to set a truncation threshold and check the corresponding gradient error; if the gradient error is much less than is required, then accordingly increase the iPOD truncation thresholds to further reduce the data storage required. Such a scenario with smaller gradient errors does not contradict the analysis presented here; the analysis holds for even extreme scenarios, but many actual computations are regular cases that yield improved behavior.
\end{remark}

\section{Conclusion}\label{Conclusion}
We proposed and analyzed an inexact gradient method based on incremental proper orthogonal decomposition to deal with the data storage difficulty in PDE-constrained optimization, and used a data assimilation problem for a detailed illustration. A step-by-step analysis is presented to demonstrate that the inexact gradient method is memory-friendly and robust. In particular, we showed that, with an appropriate iPOD setup, the inexact gradient method is expected to achieve the same level of accuracy of the optimal solution while using a similar number of gradient iterations  compared to the usual gradient method. We considered a data assimilation problem for the demonstration of the key ideas. The extension of this work to other types of PDE-constrained optimization problems is interesting future work.

\bmhead{Acknowledgements} Xuejian Li is partially supported by National Science Foundation grants DMS-1722647 and DMS-2111421. John Singler is partially supported by National Science Foundation grant DMS-2111421. Xiaoming He is partially supported by National Science Foundation grants DMS-1722647 and DMS-2152609.

\begin{appendices}

\section{Proof of Theorem \ref{thm1}}\label{ThmHeat}
\begin{proof}
	To start, we write out the backward equation (\ref{eqnback}) equipped with the approximated data $\{{u}^{j}_{h,r}\}_{j=1}^n$ instead of $\{{u}^{j}_{h}\}_{j=1}^n$:
	\begin{equation}
		\label{eqnbackpod}
		\left\{\begin{aligned}
			&-\frac{ {u}_{h,r}^{*j+1}-{u}_{h,r}^{*j}}{\tau}+A^*{u}_{h,r}^{*j}=\hat {u}^{j+1}-{u}^{j+1}_{h,r}\quad \text{in~}V_h',\\ 
			&{u}_h^{*n}= 0\quad \text{in~} L^2(\Omega),
		\end{aligned}\right.
	\end{equation}
	for $j=n-1,\cdots,1,0$. Here, ${u}_{h,r}^{*j}$ represents the solution of (\ref{eqnbackpod}) at time $t_j$ with approximated data ${u}^{j+1}_{h,r}$.\\
	Let ${e}_{h}^{j}={u}_{h,r}^{*j}-{u}_{h}^{*j}$. Subtracting (\ref{eqnbackpod}) from (\ref{eqnback}) and writing the resulting error equation in weak form gives
	\iffalse
	\begin{equation}
		\label{eqnerror}
		\left\{\begin{aligned}
			&-\frac{ {e}_{h}^{j+1}-{e}_{h}^{j}}{\tau}+A^*{e}_{h}^{j}={u}^{j+1}_{h}-{u}^{j+1}_{h,r},\\
			&{e}_h^{n}= 0.
		\end{aligned}\right.
	\end{equation}
	%where ${e}_{h}^{j}={u}_{h,r}^{*j}-{u}_{h}^{*j}$.\\
	In sense of weak formulation, equation (\ref{eqnerror}) can be equivalently written as:
	\fi
	\begin{equation}
		\label{eqnequ}
		\left\{\begin{aligned}
			&\left\langle-\frac{ {e}_{h}^{j+1}-{e}_{h}^{j}}{\tau},v_h\right\rangle+\left\langle A^*{e}_{h}^{j},v_h\right\rangle=\left\langle{u}^{j+1}_{h}-{u}^{j+1}_{h,r},v_h\right\rangle \quad \forall v_h\in V_h,\\
			&{e}_h^{n}= 0.
		\end{aligned}\right.
	\end{equation}
	Take $v={e}_{h}^{j}$ in (\ref{eqnequ}). Then use the identity $(a-b)a=\frac{a^2-b^2}{2}+\frac{(a-b)^2}{2}$ on the term $\left\langle-\frac{ {e}_{h}^{j+1}-{e}_{h}^{j}}{\tau},{e}_{h}^{j}\right\rangle$ and the property 
	$\tau\left\langle A^*{e}_{h}^{j},{e}_{h}^{j}\right\rangle+C_1\tau \|{e}_{h}^{j}\|_{L^2(\Omega)}^2\geq \tau C_2\| {e}_{h}^{j}\|_{H^1(\Omega)}^2$ to obtain 
	\begin{equation}
		\begin{split}\label{errorsec}
			&\frac{\|e_h^{j}\|_{L^2(\Omega)}^2-\|e_h^{j+1}\|_{L^2(\Omega)}^2}{2}+\frac{\|e_h^{j}-e_h^{j+1}\|_{L^2(\Omega)}^2}{2}+\tau C_2\| {e}_{h}^{j}\|_{H^1(\Omega)}^2\\&\leq \tau C_1\|{e}_{h}^{j}\|_{L^2(\Omega)}^2+\tau\left\langle{u}^{j+1}_{h}-{u}^{j+1}_{h,r},{e}_{h}^{j}\right\rangle.
		\end{split} 
	\end{equation} 
	Applying standard inequalities, such as the Cauchy-Schwartz, Poincar\'e, and Young's inequalities, yields the following bound for the last term in the right side of (\ref{errorsec}):
	\begin{equation}\label{Inequality1}
		\begin{split}
			\tau\left\langle{u}^{j+1}_{h}-{u}^{j+1}_{h,r},{e}_{h}^{j}\right\rangle&\leq\tau \|{u}^{j+1}_{h}-{u}^{j+1}_{h,r}\|_{L^2(\Omega)}\|{e}_{h}^{j}\|_{L^2(\Omega)}\\&\leq C_P\tau\|{u}^{j+1}_{h}-{u}^{j+1}_{h,r}\|_{L^2(\Omega)}\|{e}_{h}^{j}\|_{H^1(\Omega)}\\
			&\leq \frac{C_2}{2}\tau\|{e}_{h}^{j}\|_{H^1(\Omega)}^2+\frac{C_P^2}{2C_2}\tau\|{u}^{j+1}_{h}-{u}^{j+1}_{h,r}\|_{L^2(\Omega)}^2.
		\end{split} 
	\end{equation} 
	
	\noindent Combining (\ref{errorsec}) and (\ref{Inequality1}) leads to
	\begin{equation}\label{Inequality2}
		\begin{split}
			&\frac{\|e_h^{j}\|_{L^2(\Omega)}^2-\|e_h^{j+1}\|_{L^2(\Omega)}^2}{2}+\frac{\|e_h^{j}-e_h^{j+1}\|_{L^2(\Omega)}^2}{2}+\frac{C_2}{2}\tau\|{e}_{h}^{j}\|_{H^1(\Omega)}^2\\&\leq C_1\tau\|{e}_{h}^{j}\|_{L^2(\Omega)}^2+\frac{C_P^2}{2C_2}\tau\|{u}^{j+1}_{h}-{u}^{j+1}_{h,r}\|_{L^2(\Omega)}^2.
		\end{split} 
	\end{equation} 
	Summing (\ref{Inequality2}) over $j=n-1,n-2,...,2,1,0$ gives us
	\begin{equation}\label{In3}
		\begin{split}
			&\|e_h^{0}\|_{L^2(\Omega)}^2+C_2\tau\sum_{j=n-1}^{0}\|{e}_{h}^{j}\|_{H^1(\Omega)}^2+\sum_{j=n-1}^{0}\|e_h^{j}-e_h^{j+1}\|_{L^2(\Omega)}^2\\&\leq 2C_1\tau\sum_{j=n-1}^{0}\|{e}_{h}^{j}\|_{L^2(\Omega)}^2+\frac{C_P^2}{C_2}\sum_{j=n-1}^{0}\tau\|{u}^{j+1}_{h}-{u}^{j+1}_{h,r}\|_{L^2(\Omega)}^2.
		\end{split} 
	\end{equation} 
	Denote $C_3=\frac{2C_1}{1-2C_1\tau}$. Using the assumption $1-2C_1\tau>0$ and the discrete Gronwall inequality on (\ref{In3}) results in
	\begin{equation}\label{In4}
		\begin{split}
			&\|e_h^0\|_{L^2(\Omega)}^2+C_2\tau\sum_{j=n-1}^{0}\|{e}_{h}^{j}\|_{H^1(\Omega)}^2+\sum_{j=n-1}^{0}\|e_h^{j}-e_h^{j+1}\|_{L^2(\Omega)}^2\\&\leq \frac{e^{C_3T}C_P^2}{C_2}\sum_{j=n-1}^{0}\tau\|{u}^{j+1}_{h}-{u}^{j+1}_{h,r}\|_{L^2(\Omega)}^2.
		\end{split} 
	\end{equation} 
	From (\ref{gradient}), we note $\nabla J({u}^{}_{0,h})=-{u}^{*0}_{h}+\gamma  {u}^{}_{0,h}=-{u}^{*0}_{h,r}+{u}_{h,r}^{*0}-{u}_{h}^{*0}+\gamma  {u}^{}_{0,h}=-{u}^{*0}_{h,r}+e_h^0+\gamma  {u}^{}_{0,h}$, which says that $e_h^0$ is indeed the gradient error. From POD theory, we know $\left(\sum_{j=n-1}^{0}\tau\|{u}^{j+1}_{h}-{u}^{j+1}_{h,r}\|_{L^2(\Omega)}^2\right)^{\frac{1}{2}}=\left(\sum_{j=n-1}^{0}\|\tau^{\frac{1}{2}}\vec{u}^{j+1}_{h}-\tau^{\frac{1}{2}}\vec{u}^{j+1}_{h,r}\|_{\mathbb{R}^m_M}^2\right)^{\frac{1}{2}}$ $=\|U-\widetilde{U}\|_{HS}$ with  $U: \mathbb{R}^{n}\mapsto  \mathbb{R}^m_M$ and the matrix $M$ is induced by the $L^2(\Omega)$ inner product.
	Thus, with (\ref{In4}), we have
	\begin{align*}\label{fpod}
		\|\xi\|_{L^2(\Omega)}=\|e_h^0\|_{L^2(\Omega)}& \leq  C_P\sqrt{\frac{e^{C_3T}}{C_2}}\left(\sum_{j=n-1}^{0}\|\tau^{\frac{1}{2}}\vec{u}^{j+1}_{h}-\tau^{\frac{1}{2}}\vec{u}^{j+1}_{h,r}\|_{\mathbb{R}^m_M}^2\right)^{\frac{1}{2}}\\&=C_P\sqrt{\frac{e^{C_3T}}{C_2}}\|U-\widetilde{U}\|_{HS}\leq C_P\sqrt{\frac{e^{C_3T}}{C_2}}\epsilon,
	\end{align*}
	which completes the proof. 
\end{proof}

\section{Proof of Theorem \ref{pl}}\label{AP1}

\begin{proof}
	Using the L-descent inequality (\ref{Lcondition}) and the inexact gradient method (\ref{ie4}), we deduce
	\begin{equation}\label{A1}
		\begin{split}
			&J(x^{(i+1)})\leq J(x^{(i)})+\left\langle J'(x^{(i)}),x^{(i+1)}-x^{(i)}\right\rangle+\frac{L}{2}\|x^{(i+1)}-x^{(i)}\|_X^2\\&=J(x^{(i)})+\left(\nabla J(x^{(i)}),x^{(i+1)}-x^{(i)}\right)_X+\frac{L}{2}\|x^{(i+1)}-x^{(i)}\|_X^2\\
			&=J(x^{(i)})+\left(\nabla J(x^{(i)}),-\kappa(\nabla J(x^{(i)})+\xi^{(i)})\right)_X+\frac{L\kappa^2}{2}\|\nabla J(x^{(i)})+\xi^{(i)}\|_X^2\\
			&=J(x^{(i)})-\kappa\|\nabla J(x^{(i)})\|_X^2-\kappa\left(\nabla J(x^{(i)}),\xi^{(i)}\right)_X+\frac{L\kappa^2}{2}\|\nabla J(x^{(i)})\|_X^2\\&~~~~+L\kappa^2\left(\nabla J(x^{(i)}),\xi^{(i)}\right)+\frac{L\kappa^2}{2}\|\xi^{(i)}\|_X^2\\
			&=J(x^{(i)})-(\kappa-\frac{L\kappa^2}{2})\|\nabla J(x^{(i)})\|_X^2+\left(L\kappa^2-\kappa\right)\left(\nabla J(x^{(i)}),\xi^{(i)}\right)_X\\&~~~~+\frac{L\kappa^2}{2}\|\xi^{(i)}\|_X^2.
		\end{split}
	\end{equation}
	We discuss the inequality (\ref{A1}) for two cases based on $\kappa$. For $\kappa=\frac{1}{L}$, we have from (\ref{A1})
	\begin{align}\label{A2}
		J(x^{(i+1)})&\leq J(x^{(i)})-\frac{1}{2L}\|\nabla J(x^{(i)})\|_X^2+\frac{1}{2L}\|\xi^{(i)}\|_X^2.
	\end{align}
	Applying the $\mu$-{\bf Polyak-Lojasiewicz} inequality (\ref{PL})  and subtracting $J(x^{\star})$ on both sides of (\ref{A2}), we obtain 
	\begin{align}\label{A4}
		J(x^{(i+1)})-J(x^{\star})&\leq \left(1-\frac{\mu}{L}\right)\left(J(x^{(i)})-J(x^{\star})\right)+\frac{1}{2L}\|\xi^{(i)}\|_X^2.
	\end{align}
	Unrolling the recursion (\ref{A4}) gives us 
	\begin{align}\label{A5}
		J(x^{(k)})-J(x^{\star})&\leq \left(1-\frac{\mu}{L}\right)^{k}\left(J(x^{(0)})-J(x^{\star})\right)+\sum_{i=0}^{k-1}\frac{1}{2L}\|\xi^{(i)}\|_X^2\left(1-\frac{\mu}{L}\right)^{i}\\
		&\leq \left(1-\frac{\mu}{L}\right)^{k}\left(J(x^{(0)})-J(x^{\star})\right)+\frac{1}{2\mu}\left(1-\left(1-\frac{\mu}{L}\right)^k\right)\epsilon^2.
	\end{align}
	This completes the first part of Theorem \ref{pl}.
	
	Second, we consider the case: $0<\kappa<\frac{1}{L}$ and $0<\mu (2\kappa-L\kappa^2)<1$.
	Applying the Cauchy-Schwartz inequality, the $\mu$-{\bf Polyak-Lojasiewicz} inequality (\ref{PL}), and the inequality (\ref{gradientbound}) from Lemma \ref{covv} on (\ref{A1}), we obtain
	\begin{equation}\label{A55}
		\begin{split}
			J(x^{(i+1)})&\leq J(x^{(i)})-\left(\kappa-\frac{L\kappa^2}{2}\right)\|\nabla J(x^{(i)})\|_X^2\\&+|L\kappa^2-\kappa|\|\nabla J(x^{(i)})\|_X\|\xi^{(i)}\|_X+\frac{L\kappa^2}{2}\|\xi^{(i)}\|_X^2\\
			&\leq J(x^{(i)})-\mu(2\kappa-L\kappa^2)\left(J(x^{(i)})-J(x^{\star})\right)\\&+\sqrt{2L}|L\kappa^2-\kappa|\sqrt{J(x^{(i)})-J(x^{\star})}\|\xi^{(i)}\|_X+\frac{L\kappa^2}{2}\|\xi^{(i)}\|_X^2.
			%&\leq J(x^{(i)})-\mu(2\kappa-L\kappa^2)\left(J(x^{(i)})-J(x^{\star})\right)+\left(J(x^{(i)})-J(x^{\star})\right)\|\xi^{(i)}\|_X+\frac{\mu(L\kappa^2-\kappa)^2}{2}\|\xi^{(i)}\|_X+\frac{L\kappa^2}{2}\|\xi^{(i)}\|_X^2
			%&\leq J(x^i)-\left(\kappa-\frac{L\kappa^2}{2}\right)\left(J(x^{(i)})-J(x^{\star})\right)+\frac{\rho(L\kappa^2-\kappa)^2}{4}\|\xi^{(i)}\|_X+\frac{L\kappa^2}{2}\|\xi^{(i)}\|_X^2.
		\end{split}
	\end{equation}
	Denote $\theta=1-\mu(2\kappa-L\kappa^2)$. Since we assume $ J $ is L-descent, $0<\kappa<1/L$, and $\|\xi^{(i)}\|_X<\frac{2-\kappa L}{4-\kappa L}\|\nabla J(x^{(i)})\|_X$, Theorem \ref{thm4} tells us that the function $J $ is decreasing. Then subtracting $J(x^{\star})$ on both sides of (\ref{A55}) and using that $J$ is decreasing along iterations gives
	\begin{equation}\label{A6}
		\begin{split}
			J(x^{(i+1)})-J(x^{\star})&\leq \sqrt{2L}|L\kappa^2-\kappa|\sqrt{J(x^{(0)})-J(x^{\star})}\|\xi^{(i)}\|_X+\frac{L\kappa^2}{2}\|\xi^{(i)}\|_X^2\\&~~~~+\theta\left(J(x^{(i)})-J(x^{\star})\right).
		\end{split}
	\end{equation}
	Unrolling the recursion (\ref{A6}) gives us 
	\begin{equation*}\label{A7}
		\begin{split}
			&J(x^{(k)})-J(x^{\star})\leq \theta^{k}\left(J(x^{(0)})-J(x^{\star})\right)\\&+\sum_{i=0}^{k-1}\left(\sqrt{2L}|L\kappa^2-\kappa|\sqrt{J(x^{(0)})-J(x^{\star})}\|\xi^{(i)}\|_X+\frac{L\kappa^2}{2}\|\xi^{(i)}\|_X^2\right)\theta^{i}\\
			&\leq \theta^{k}\left(J(x^{(0)})-J(x^{\star})\right)+\frac{1-\theta^k}{1-\theta}\left(\sqrt{2L}|L\kappa^2-\kappa|\sqrt{J(x^{(0)})-J(x^{\star})}\epsilon+\frac{L\kappa^2}{2}\epsilon^2\right),
		\end{split}
	\end{equation*}
	which completes the proof. 
\end{proof}
\section{Proof of Theorem \ref{sc}}\label{AP2}
\begin{proof}
	Based on the inexact gradient method (\ref{ie4}), we have
	\begin{equation}\label{BB1}
		\begin{split}
			&\|x^{(i+1)}-x^\star\|_X^2=\|x^{(i)}-\kappa\left(\nabla J(x^{(i)})+\xi^{(i)}\right)-x^\star\|_x^2\\
			%&=\|x^{(i)}-x^\star\|_X^2-2\kappa\left(\nabla J(x^{(i)})+\xi^{(i)},x^{(i)}-x^\star\right)_X+\kappa^2\|\nabla J(x^{(i)})+\xi^{(i)}\|_X^2\\
			&=\|x^{(i)}-x^\star\|_X^2-2\kappa\left(\nabla J(x^{(i)}),x^{(i)}-x^\star\right)_X-2\kappa\left(\xi^{(i)},x^{(i)}-x^\star\right)_X\\&~~~~+\kappa^2\|\nabla J(x^{(i)})\|_X^2+2\kappa\left(\nabla J(x^{(i)}),\xi^{(i)}\right)+\kappa^2\|\xi^{(i)}\|_X^2.
		\end{split}
	\end{equation}
	Recall $0<\kappa<\frac{1}{L}$, then there exists a constant $\eta>0$ satisfying $2\kappa-2\kappa^2L-\eta=0$. Using inequality (\ref{SC}), (\ref{gradientbound}) from Lemma \ref{cov}, the Cauchy-Schwartz inequality, and Young's inequality we bound (\ref{BB1}) as:
	\begin{equation}\label{B1}
		\begin{split}
			&\|x^{(i+1)}-x^\star\|_X^2
			\leq \|x^{(i)}-x^\star\|_X^2-2\kappa\left(J(x^{(i)})-J(x^{\star})+\frac{\mu}{2}\|x^{(i)}-x^\star\|_X^2\right)\\&+2\kappa\|x^{(i)}-x^\star\|_X\|\xi^{(i)}\|_X+2\kappa^2L\left(J(x^{(i)})-J(x^{\star})\right)+\eta\left(J(x^{(i)})-J(x^{\star})\right)\\&+\frac{\kappa^2 }{2L\eta}\|\xi^{(i)}\|_X^2+\kappa^2\|\xi^{(i)}\|_X^2\\
			&= \|x^{(i)}-x^\star\|_X^2-\left(\kappa\mu\|x^{(i)}-x^\star\|_X^2-2\kappa\|x^{(i)}-x^\star\|_X\|\xi^{(i)}\|_X\right.\\&\left.-\frac{\kappa^2 }{2L\eta}\|\xi^{(i)}\|_X^2-\kappa^2\|\xi^{(i)}\|_X^2\right)
		\end{split}
	\end{equation}
	%Recall  is the termination condition of iterations. 
	Recall $\|\xi^{(i)}\|_X\leq \epsilon$ for all $i$ and $\delta\geq\frac{\sqrt{2L\eta}+\sqrt{2L\eta+2L\eta\kappa\mu +\kappa\mu}}{\mu\sqrt{2L\eta}}\epsilon$. Since the iteration is not terminated, i.e., $ \|x^{(i)}-x^\star\|_X\geq\delta$, we then have $-\kappa\mu\|x^{(i)}-x^\star\|_X^2+2\kappa\|x^{(i)}-x^\star\|_X\|\xi^{(i)}\|_X+\frac{8\kappa^2 L}{\eta}\|\xi^{(i)}\|_X^2+\kappa^2\|\xi^{(i)}\|_X^2<0$, and the sequence $\{\|x^{(i)}-x^\star\|_X\}_{i\geq 0}$ is non-increasing. %for all $i$ satisfying  $\|x^{(i)}-x^\star\|_X\geq  \delta$. 
	Hence, we bound (\ref{B1}) as 
	\begin{equation}\label{B4}
		\begin{split}
			\|x^{(i+1)}-x^\star\|_X^2&\leq (1-\kappa\mu)\|x^{(i)}-x^\star\|_X^2+2\kappa\|x^{(0)}-x^\star\|_X\|\xi^{(i)}\|_X\\&+\frac{\kappa^2 }{2L\eta}\|\xi^{(i)}\|_X^2+\kappa^2\|\xi^{(i)}\|_X^2.
		\end{split}
	\end{equation}
	Denote $\theta=1-\mu\kappa$. %Choosing $\kappa$ such that $0<\theta<1$ and 
	Unrolling the recursion (\ref{B1}) gives us 
	\begin{align*}
		&\|x^{(k)}-x^\star\|_X^2\\&\leq \theta^k\|x^{(0)}-x^\star\|_x^2+\sum_{i=0}^{k-1}\theta^{i}\left(2\kappa\|x^{(0)}-x^\star\|_X\|\xi^{(i)}\|_X+\frac{\kappa^2 }{2L\eta}\|\xi^{(i)}\|_X^2+\kappa^2\|\xi^{(i)}\|_X^2\right)\\
		&\leq\theta^k\|x^{(0)}-x^\star\|_x^2+\frac{1-\theta^k}{1-\theta}\left(2\kappa\|x^{(0)}-x^\star\|_X\epsilon+\frac{\kappa^2 }{2L\eta}\epsilon^2+\kappa^2\epsilon^2\right),
	\end{align*}
	which completes the proof. 
\end{proof}
\iffalse
\section{Proof of Theorem \ref{scd}}\label{AP3}
\begin{proof}
	It is easy to verify that for $\mu$-strong convex and twice differentiable function with $\|\nabla^2 J(x)\|\leq L$, all  eigenvalues of $\nabla^2 J(x)$ are positive and in-between $\mu$ and $L$. We thus have
	\begin{equation}\label{CC1}
		\begin{split}
			\|x^{(i+1)}-x^\star\|_X&=\|x^{(i)}-\kappa\left(\nabla J(x^{(i)})+\xi^{(i)}\right)-x^\star\|_X\\&=\|x^{(i)}-x^\star-\kappa\left(\nabla J(x^{(i)})-\nabla J(x^\star)\right)-\kappa\xi^{(i)}\|_X\\
			&=\|\left(I-\nabla^2 J(\zeta)\right)(x^{(i)}-x^{\star})-\kappa \xi^{(i)}\|_X\\&\leq \|I-\nabla^2 J(\zeta)\|_{\mathscr{L}(X,X)}\|x^{(i)}-x^{\star}\|_X+\kappa\|\xi^{(i)}\|_X\\&\leq \theta \|x^{(i)}-x^{\star}\|_X+\kappa\|\xi^{(i)}\|_X.
		\end{split}
	\end{equation}
	Here, the $\zeta\in (x^{\star},x^{(i)})$ and $\theta=\max\{|1-\kappa \mu|, |1-\kappa L|\}$.\\
	Choosing $\kappa$ such that $0<\theta<1$ and unrolling the recursion (\ref{CC1}) gives us 
	\begin{align*}
		\|x^{(k)}-x^\star\|_X\leq \theta^k\|x^{(0)}-x^\star\|_X+\sum_{i=0}^{k-1}\kappa\theta^i\|\xi^{(i)}\|_X\leq\theta^k\|x^{(0)}-x^\star\|_X+\frac{1-\theta^k}{1-\theta}\kappa\epsilon,
	\end{align*}
	which complete the proof.
\end{proof}
\fi

\end{appendices}


\bibliography{data_assimilation_NEWN,sn-bibliography}% common bib file

\end{document}